\documentclass[11pt]{amsart}
\usepackage{amsmath,amssymb,amsthm,graphicx}
\usepackage[usenames]{color}
\usepackage[shortlabels]{enumitem}
\usepackage{hyperref}
\usepackage{xparse}
\usepackage{float}
\usepackage{morefloats}
\usepackage{lipsum}
\usepackage{mathtools}
\usepackage{xcolor}

\newtheorem{theorem}{Theorem}[section]
\newtheorem*{theorem-non}{Theorem}
\newtheorem{prop}[theorem]{Proposition}
\newtheorem{lemma}[theorem]{Lemma}
\newtheorem{remark}[theorem]{Remark}

\newtheorem{definition}[theorem]{Definition}
\newtheorem{cor}[theorem]{Corollary}
\newtheorem{example}[theorem]{Example}

\newtheorem{case}{Case}

\numberwithin{subcase}{case}

\newcommand{\CPb}{\overline{\mathbb{CP}}^{2}}
\newcommand{\CP}{{\mathbb{CP}}^{2}}

\def \CPb {\overline{\mathbb{CP}}^{2}}
\def \CP {{\mathbb{CP}}^{2}}

\def \s {\sigma}

\def \bd {\partial}

\def \- {\setminus}

\title[Deformation of Singular Fibers and Small Exotic Symplectic $4$-Manifolds]{Deformation of Singular Fibers of Genus $2$ Fibrations and Small Exotic Symplectic $4$-Manifolds}

\begin{document}

\author{Anar Akhmedov}
\address{School of Mathematics,
University of Minnesota,
Minneapolis, MN, 55455, USA}
\email{akhmedov@math.umn.edu}

\author{S\"{u}meyra Sakalli}
\address{Max Planck Institute for Mathematics,
Vivatsgasse 7, 53111 Bonn
Germany}
\email{sakal008@math.umn.edu}

\begin{abstract} We introduce the $2$-nodal spherical deformation of certain singular fibers of genus $2$ fibrations, and use such deformations to construct various examples of simply connected minimal symplectic $4$-manifolds with small topology. More specifically, we construct new exotic minimal symplectic $4$-manifolds homeomorphic but not diffeomorphic to $\CP\#6\CPb$, $\CP\#7\CPb$, and $3\CP\#k\CPb$ for $k=16, 17, 18, 19$ using combinations of such deformations, symplectic blowups, and (generalized) rational blowdown surgery. We also discuss generalizing our constructions to higher genus fibrations using $g$-nodal spherical deformations of certain singular fibers of genus $g \geq 3$ fibrations. \end{abstract}

\subjclass{Primary 57R55; Secondary 57R57, 57M05}

\keywords{symplectic 4-manifold, mapping class group, Lefschetz fibration, genus two fibration, , rational blowdown}

\maketitle

\section{Introduction} 

There has been lots of activity in the discovery of exotic smooth structures on simply connected $4$-manifolds with small Euler characteristic in the period of last $15$ years. The following references are among many dealing with this subject \cite{P2, SS1, FS3, PSS, SS2, P3, A1, akhmedov, AP1, ABP, ABBKP, AP2, Mich, Yeung, AkP, KS}. In this article, we will be primarily interested in the construction of exotic smooth structures on simply connected $4$-manifolds with small Euler characteristic using the rational-blowdown surgery \cite{FS1, P1}. To set the stage and motivate our work, let us briefly survey what is already known in this direction. In the 2004, J. Park \cite{P2} constructed the first known exotic smooth structure on $\CP\#7\CPb$ i.e. a smooth $4$-manifold homeomorphic but not diffeomorphic to $\CP\#7\CPb$. Park's manifold in \cite{P2} was constructed from the elliptic surface $E(1)=\CP\#9\CPb$ with a certain elliptic fibration structure by applying oridinary blowups and rational blowdown surgery. Shortly afterwards, using a technique similar to Park's (i.e. the generalized rational blowdown), A. Stipsicz and Z. Szab\'o constructed an exotic smooth structure on  $\CP\#6\CPb$ \cite{SS1}. Then Fintushel and Stern \cite{FS3} introduced a new method, the knot surgery in double nodes, which yielded to infinitely many distinct smooth structures on $\CP\#k\CPb$, for $k= 6,7,8$. Park, Stipsicz, and Szab\'o  \cite{PSS}, using \cite{FS3} and the rational blowdown, constructed infinitely many exotic smooth structures on $\CP\#5\CPb$. In \cite{SS2} and \cite{P3}, A. Stipsicz and Z. Szab\'o and J. Park used similar ideas to construct exotic smooth structures on $3\,\CP\#k\CPb$ for $k = 9$ \cite{SS2} and for $k = 8$ \cite{P3} respectively. All these infinite families of exotic $4$-manifolds were constructed from elliptic surfaces $E(1)$ and $E(2)$, with certain elliptic fibration structures on them, by applying a combination of knot surgery in double nodes, oridinary blowups and rational blowdowns. Similar results starting from the ellptic surface $E(n)$ for $n\geq 3$ were obtained in \cite{A}. See also a related more recent work in \cite{Mich, KS}, which again uses elliptic fibrations on $E(1)$. One of the key ingredients in the above mentioned articles was the use of Kodaira's  classification of singular fibers in elliptic fibrations. Another interesting approach is given in \cite{Yeung}, where various small exotic complex surfaces with $b_2^{+} = 1, 3$ are constructed from fake projective planes.

Motivated by these results, our goal in this paper is to construct exotic smooth and symplectic structures on $\CP\#k\CPb$, for $k= 6,7$ and $3\CP\#k\CPb$ for $k = 16, 17, 18, 19$ starting with certain genus two Lefschetz fibration structures on $\CP\#13\CPb$ and $E(2)\#2\CPb$. We will first introduce the $2$-nodal spherical deformation of certain singular fibers of genus $2$ fibrations, and use such deformations to construct various examples of simply connected minimal symplectic $4$-manifolds with small topology. More specifically, we construct new exotic minimal symplectic $4$-manifolds homeomorphic but not diffeomorphic to $\CP\#6\CPb$, $\CP\#7\CPb$, and $3\CP\#k\CPb$ for $k = 16, 17, 18, 19$ using such spherical deformations, symplectic blowups, and the (generalized) rational blowdown surgery. Our work relies on the geometrical classification of all singular fibers in fibrations of genus two curves by Namikawa and Ueno, and constructions introduced in our paper can be considered a genus two counterpart of the constructions given in \cite{P2, SS1, PSS, SS2, P3, akhmedov}. In addition, we would like to emphasize that our paper is first in applying Namikawa and Ueno's classification in rational blowdown constructions, and using pencils of genus two curves in Hirzebruch surfaces to construct small exotic symplectic $4$-manifolds. 

For the sake of clarity, let us outline here key results and ideas leading to our main theorems. In the article~\cite{NamU}, Namikawa and Ueno gave the complete list of singular fibers for fibrations of complex curves of genus two over the $2$-disk. More recently, in paper \cite{Gong}, Gong, C., Lu, J. and Tan, S.-L. classified fibrations of complex curves of genus two over the $2$-sphere with exactly two singular fibers. There are $11$ such fibrations and the singular fibers are coming from the Namikawa-Ueno's list. That is to say, Gong, et al. determined which two fibrations of Namikawa-Ueno can be glued to obtain a genus two fibration over the $2$-sphere. The paper~\cite{Gong} also provides the defining polynomials for each of these $11$ fibrations and lists the monodromies using the work of M. Ishizaka \cite{Ish2, Ish}. In our paper, we will work with $2$ of the $11$ families of Gong, et.al in \cite{Gong}. Namely, we consider $2$ of these fibrations of complex curves of genus $2$ over the $2$-sphere, each of them having $2$ singular fibers. In our constructions, for each family, we have determined the homology classes of components of one of the two singular fibers. To determine the homology classes in our first fibration (VIII-4, VIII-1), we use Kitagawa's work \cite{Kit1} (see Lemma 3.2 in \cite{Kit1} or Lemma~\ref{lemma1} in this paper). Using Lemma 3.2 in \cite{Kit1}, from fiber of type VIII-4, he constructs the desired pencil by contracting $(-1)$ curves consecutively, which we employ in our proof. Next, we take this pencil and by blow-ups, determine the homology classes of components of the fiber VIII-4. Its complement is VIII-1 by the theorem of Gong, et al. in \cite{Gong}. The monodromy of VIII-1 is also known by Ishizaka’s work in \cite{Ish2, Ish}. By using the monodromy and applying the deformation method to VIII-1, we then start our construction. For other fibration, we use the same strategy: first similar to Kitagawa's method, by contraction of $-1$ spheres, we obtain our pencils, and then use the same strategy as in (VIII-4, VIII-1) case. We have also benefited from the discussions of the papers ``On certain Mordell-Weil lattices of hyperelliptic type on rational surfaces" \cite{Nguen}, where successive contraction of $-1$ spheres was used to determine the homology classes and \cite{Gong2}, where the classification and Mordell-Weil groups of fibrations with two singular fibers are given. The later result determines the number of disjoint sections of fibrations that we consider.

{\bf Acknowledgments:} We would like to thank Ronald Stern for very useful discussions, dating back to first author's Ph.D. studies at University of California, Irvine. Our work is greatly motivated and inspired by some of these discussions. We also would like to thank Cagri Karakurt, Tian-Jun Li, and Sai-Kee Yeung for their interest in our work, and pointing out some references. The authors are also very grateful to the referee for a careful reading of the paper and many valuable comments and suggestions, which helped us to improve our manuscript. We have also greatly benefited from IPE drawing program (\url{http://ipe.otfried.org}) in drawing the various figures in this paper. A. Akhmedov was partially supported by NSF grants DMS-1065955, DMS-1005741, Sloan Research Fellowship, Simons Research Fellowship and Collaboration Grants for Mathematicians by Simons Foundation. S. Sakall{\i} was partially supported by NSF grants DMS-1065955.  

\section{Preliminaries and background} To make our paper self-contained, in this section we review some background materials on the rational blowdown and its generalization \cite{FS1, P1}, Hirzebruch surfaces \cite{H}, classification of the singular fibers of genus two fibration due to Namikawa and Ueno \cite{NamU}, and list a few results that are to be used in the later part of the paper.

\subsection{Rational blow-down and its generalizations}\label{rb} The rational blow-down surgery was introduced in \cite{FS1}. The basic idea of the surgery is that if a smooth 4-manifold $X$ contains a particular configuration $C_{p}$ of transversally intersecting $2$-spheres whose boundary is the lens space $L(p^2,1-p)$, then one can replace $C_{p}$ with the rational homology ball $\mathbb{B}_{p}$ to construct a new manifold $X_{p}$. If one knows the Seiberg-Witten invariants of the original manifold $X$, then one can determine the Seiberg-Witten invariants of $X_{p}$. The rational blow-down surgery technique was generalized in \cite{P1}. Since we will be also using the generalized rational blow-down in our constructions, let us review the generalized rational blow-down below. Let $p\geq q\geq 1$ and $p,q$ be relatively prime integers. Let  $C_{p,q}$ denote the smooth $4$-manifold obtained by plumbing disk bundles over the $2$-sphere according to the following linear diagram

\begin{picture}(100,60)(-90,-25)
 \put(-12,3){\makebox(200,20)[bl]{$-r_{k}$ \hspace{6pt}
                                  $-r_{k-1}$ \hspace{96pt} $-r_{1}$}}
 \put(4,-25){\makebox(200,20)[tl]{$u_{k}$ \hspace{25pt}
                                  $u_{k-1}$ \hspace{86pt} $u_{1}$}}
  \multiput(10,0)(40,0){2}{\line(1,0){40}}
  \multiput(10,0)(40,0){2}{\circle*{3}}
  \multiput(100,0)(5,0){4}{\makebox(0,0){$\cdots$}}
  \put(125,0){\line(1,0){40}}
  \put(165,0){\circle*{3}}
\end{picture}

\noindent where $p^2/(pq-1) = [r_k, r_{k-1}, \cdots, r_1]$ is the unique continued linear fraction with all $r_{i} \geq 2$ and each vertex $u_{i}$ of the linear diagram represents a disk bundle over the 2-sphere with Euler number $-r_{i}$. According to Casson and Harer \cite {CH}, the boundary of $C_{p,q}$ is the lens space $L(p^2, 1-pq)$ which also bounds a rational ball ${\mathbb{B}}_{p,q}$ with $\pi_1({\mathbb{B}}_{p,q})=\mathbb{Z}_p$ and $\pi_1(\bd {\mathbb{B}}_{p,q})\to \pi_1({\mathbb{B}}_{p,q})$ is surjective. If $C_{p,q}$ is embedded in the $4$-manifold $X$ then the generalized rational blowdown manifold $X_{p,q}$  is obtained by replacing $C_{p,q}$ with ${\mathbb{B}}_{p,q}$, i.e, $X_{p,q} = (X\- C_{p,q}) \cup {\mathbb{B}}_{p,q}$. If $X$ and $X\- C_{p,q}$ are simply connected, then so is $X_{p,q}$. The case when $q=1$ is the construction of Fintushel-Stern with $C_p = C_{p,1}$ given by

 \begin{picture}(100,60)(-90,-25)
 \put(-12,3){\makebox(200,20)[bl]{$-(p+2)$ \hspace{6pt}
                                  $-2$ \hspace{96pt} $-2$}}
 \put(4,-25){\makebox(200,20)[tl]{$u_{p-1}$ \hspace{25pt}
                                  $u_{p-2}$ \hspace{86pt} $u_{1}$}}
  \multiput(10,0)(40,0){2}{\line(1,0){40}}
  \multiput(10,0)(40,0){2}{\circle*{3}}
  \multiput(100,0)(5,0){4}{\makebox(0,0){$\cdots$}}
  \put(125,0){\line(1,0){40}}
  \put(165,0){\circle*{3}}
\end{picture}

\smallskip 

\begin{lemma}\label{rbd} Let $X_{p,q}$ be the smooth $4$-manifold obtained from $X$ by a rational blow-down of the configuration $C_{p,q}$. Then ${b_{2}}^{+}(X_{p,q}) = {b_{2}}^{+}(X)$, ${b_{2}}^{-}(X_{p,q}) = {b_{2}}^{-}(X) - k$, $e(X_{p,q}) = e(X) - k$, and ${c_{1}}^{2}(X_{p,q}) = {c_{1}}^{2}(X) + k$.

\end{lemma}

\begin{proof} 

Since $C_{p,q}$ is a negative definite plumbing of length $k$, we have ${b_{2}}^{+}(X_{p,q}) = {b_{2}}^{+}(X)$, ${b_{2}}^{-}(X_{p,q}) = {b_{2}}^{-}(X) - k$, and consequently $e(X_{p,q}) = e(X) - k$. Using the formula ${c_1}^{2} := 3\s +2e$, we compute ${c_{1}}^{2}(X_{p,q}) = 3\s(X_{p,q}) + 2e(X_{p,q}) = 3(\s(X)+k) + 2(e(X)-k) = {c_{1}}^{2}(X) + k$. 

\end{proof}

The following theorem gives a way to compute the Seiberg-Witten invariants of $X_{p,q}$ using the Seiberg-Witten invariants of $X$.

\begin{theorem} \cite {P1}. Suppose $X$ is a smooth 4-manifold with $b_{2}^{+}(X) > 1$ which contains a configuration $C_{p,q}$. If $L$ is a characteristic line bundle on $X$ such that, $SW_{X}(L) \ne 0$, $(L|_{C_{p,q}})^{2} = - b_{2}(C_{p,q})$ and $c_{1}(L|_{L_(p^2,1-pq)}) = mp \in {\mathbb{Z}}_{p^{2}} \cong H^{2}(L(p^2, 1-pq); \mathbb{Z})$ with $m \equiv (p-1) \mod 2$, then $L$ induces a SW basic class $\bar L$ of $X_{p,q}$ such that $SW_{X_{p,q}}(\bar L) = SW_{X}(L)$. \end{theorem}

\medskip

\begin{cor} \cite {P1}. Suppose $X$ is a smooth 4-manifold with $b_{2}^{+}(X) > 1$ which contains a configuration $C_{p,q}$. If $L$ is a SW basic class of $X$ satisfying $L\cdot u_{i} = (r_{i} - 2)$ for any i with $1 \leq i \leq k$ (or  $L\cdot u_{i} = -(r_{i} - 2)$, then $L$ induces a SW basic class $\bar L$ of $X_{p,q}$ such that $SW_{X_{p,q}}(\bar L) = SW_{X}(L)$.  

\end{cor}

The algebraic geometric interpretation of the rational blowdown procedure can be found in \cite{KSB}, which is very useful in the construction of exotic complex surfaces. It seems quite promising to prove some of the exotic manifolds constructed in our paper are complex algebraic surfaces of general type. We plan to investigate this in the future work. 

The following result by T-J. Li and A-K. Liu \cite{LL} will also be needed in the sequel.

\begin{theorem} \cite {LL}
There is a unique symplectic structure on $\CP \# k \CPb$ for $2 \leq k \leq 9$ up to diffeomorphisms and deformation. For $k \leq 10$, the symplectic structure is still unique for the standard canonical class.  
\end{theorem}

\subsection{Hirzebruch Surfaces} \label{HS}

In this section, we define and review some basic facts and properties of Hirzebruch surfaces ${\mathbb{F}}_{n} = {\mathbb{P}}_{{\mathbb{P}}_{1}} (O \oplus O(-n))$ which will be needed in the sequel. The reader is referred to the reference \cite{H} for additional details.

A convenient way to define Hirzebruch surfaces ${\mathbb{F}}_{n}$ is as follows. For any non-negative integer $n$, we define $\mathbb{F}_{n} =\big\{((X_{0} :X_{1} :X_{2}),(Y_{0} :Y_{1}))|X_{1}{Y_1}^{n} = X_{2}{Y_0}^{n} \big \} \subset \mathbb{CP}^{2} \times \mathbb{CP}^1$ and call it the Hirzebruch surface of degree $n$. Notice that the restriction of the second projection to $\mathbb{F}_{n}$ gives a structure of a $\mathbb{CP}^1$-bundle. The complex surfaces $\mathbb{F}_n$ are holomorphic $\mathbb{CP}^1$-bundles over $\mathbb{CP}^1$ with holomorphic sections of self-intersection $-n$ (i.e. they are geometrically ruled complex surfaces). Conversely, any holomorphic $\mathbb{CP}^1$ bundle over  $\mathbb{CP}^1$ is isomorphic to $\mathbb{F}_n$ for some $n$. Let us consider a Zariski open set defined by $X_0 Y_0 \neq 0$ and take $(x,y) = (X_{1}/X_{0},Y_{1}/Y_{0})$ as affine coordinates. $\mathbb{F}_{n}$ has a minimal section defined by $x = 0$ and the fibre defined by $y = 0$. 


The following facts are well known and easy to prove:\\ 
\begin{enumerate}

\item $\mathbb{F}_0 \simeq \mathbb{CP}^1 \times \mathbb{CP}^1 = \mathbb{S}^2 \times \mathbb{S}^2$,\\ 

\item $\mathbb{F}_1 \simeq \mathbb{CP}^2 \# \CPb = \mathbb{S}^2 \tilde \times \mathbb{S}^2$,\\ 

\item $\mathbb{F}_n$ is a minimal complex surface if and only if $n \neq 1$,\\ 

\end{enumerate}

As smooth 4-manifolds, $\mathbb{F}_n$ is diffeomorphic to $\mathbb{F}_m$ if and only if $n \equiv m\ (mod \ 2)$. Moreover, smooth 4-manifolds $\mathbb{S}^2 \times \mathbb{S}^2$ and $ \mathbb{CP}^2 \# \CPb$ admit infinitely many inequivalent complex strucutres. Indeed, as a complex manifold, $\mathbb{F}_n$ is complex diffeomorphic to $\mathbb{F}_m$ if and only if $n = m$ (\cite{H}). 







The $n$-th Hirzebruch surface $\mathbb{F}_n$ admits two disjoint holomorphic sections of self intersections $n$ and $-n$. Let us denote them by $C_{\infty}$ and $C_0$ respectively, and the fiber class of $\mathbb{F}_n$ by $F$. It is easy to verify that $C_0 = C_{\infty}-nF$ (for the proof see \cite{Be}, Proposition IV.1, part (ii), page 40). This Proposition also shows that for any $n > 0$ there is a unique irreducible curve $C$ on $\mathbb{F}_n$ with negative self-intersection and class $C = C_{\infty}-nF$.

To determine the canonical class $K_{\mathbb{F}_n}$ of $\mathbb{F}_n$, let us express  $K_{\mathbb{F}_n}$ as $K_{\mathbb{F}_n}=aC_{\infty}+bF$. By applying the adjunction formula to classes $F$ and then $C_{\infty}$, we find 
\begin{equation}
K_{\mathbb{F}_n}= -2C_{\infty} + (n-2)F.
\end{equation}  
In particular, $K_{\mathbb{F}_2}= -2C_{\infty}$ and $K_{\mathbb{F}_3}= -2C_{\infty}+F$. The proof of the above formula for $K_{\mathbb{F}_n}$ can also be found in \cite{Hart} (see Corollary 2.11, page 374), though notation in \cite{Hart} is slightly different than ours; $C_{\infty}$ section there is denoted by $C_0$. By setting $g=0$ in Corollary 2.11 and applying the analogue of the formula $C_0 = C_{\infty}-nF$ gives the formula for $K_{\mathbb{F}_n}$. 

There is a well-known diffeomorphism between $\mathbb{F}_2 \# \CPb = \mathbb{S}^2 \times \mathbb{S}^2 \# \CPb$ and $\mathbb{S}^2 \tilde \times \mathbb{S}^2 \# \CPb = \mathbb{CP}^2 \# 2\CPb$, which can be verified by applying the sequence of $2$-handle moves as in Figure~\ref{dttt}.

\begin{figure}
\begin{center}
\scalebox{0.58}{\includegraphics{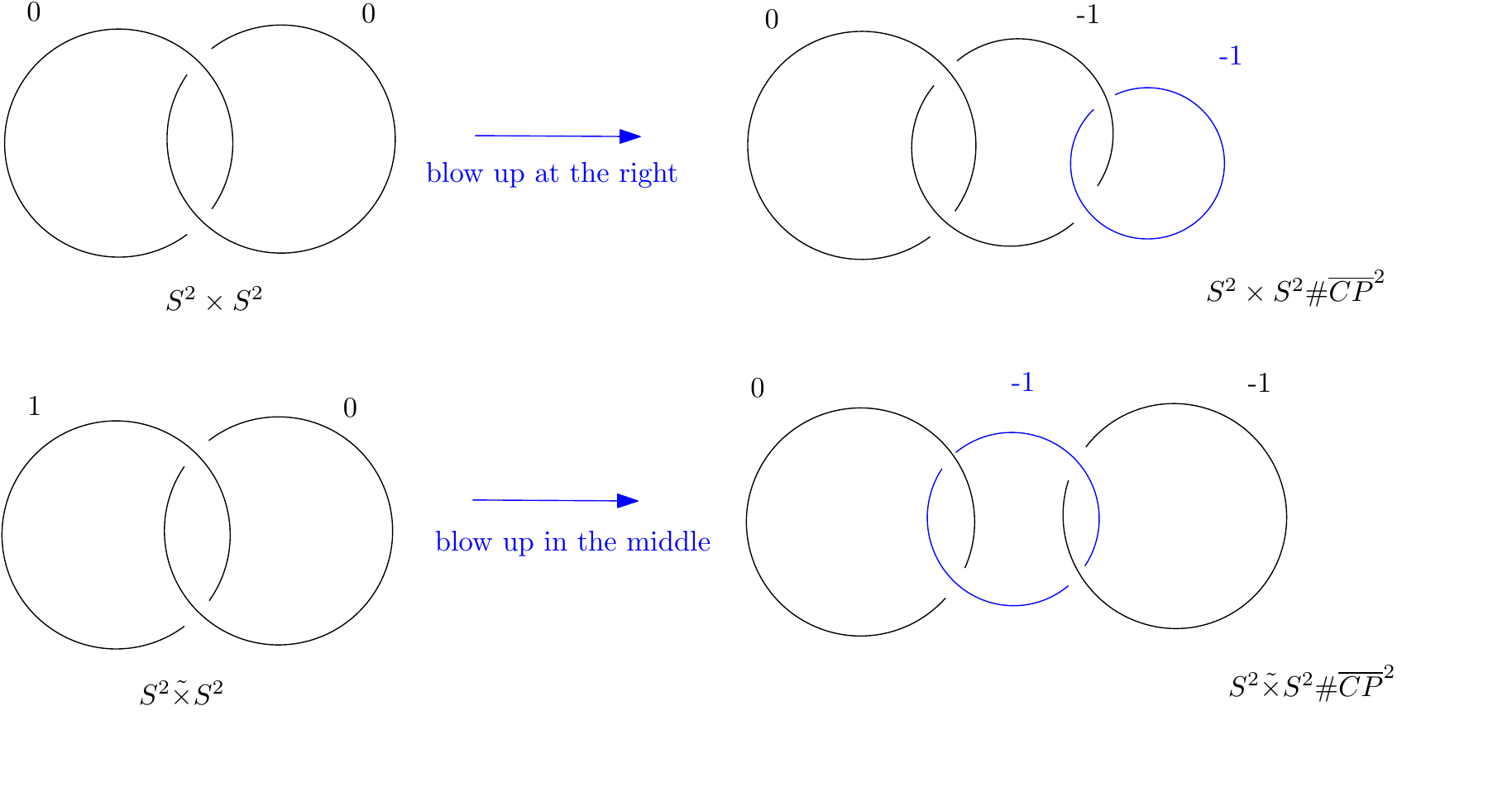}}
\caption{}
\label{dttt}
\end{center}
\end{figure}
 
In what follows, we will write down some explicit classes that will be needed later on in our computations. In $\mathbb{F}_2 \# \CPb $ let us take classes $F$, $C_{\infty}$, $C_0$ and class $e$ of the exceptional divisor coming from the blow up, where $F^2=0$, $C_{\infty}^2 = 2$, $C_0^2 = -2$, $e^2=-1$ and $C_{\infty} \cdot C_0 = 0, C_{\infty} \cdot F =1, C_0 \cdot F =1$. For computational purposes, we will write them in terms of classes $h$, $e_1$ and $e_2$ of squares $1$, $-1$ and $-1$ in the diffeomorphic manifold $\mathbb{CP}^2 \# 2\CPb$. Firstly, the canonical class $K$ of $\mathbb{CP}^2 \# 2\CPb$ is $-3h+e_1+e_2$, which  follows from the blow up formula. Let $F = ah+be_1+ce_2$. By solving equations $F^2 =0$, using adjunction equality for $F$, and assuming one of the two of equations $F \cdot e_1 = 1$ or $F \cdot e_2 = 1$, arising from the base point locus (see also the Figures~\ref{cl1} and ~\ref{cl2} for the illustration), we determine $a$, $b$ and $c$. We have the following two possibilities:
\begin{equation}
F=h-e_1 \ or \  F=h-e_2.
\end{equation}  

By fixing which exceptional divisor intersects the fiber, without loss of generality, we can assume that $F=h-e_1$. In the same way, we find
\begin{equation}
C_{\infty} = 2h-e_1-e_2, \;\;C_0=e_1 - e_2, \;\; e= h-e_1-e_2.
\end{equation}  

\begin{remark}\label{crm} There is a more convenient geometric way to see the above classes in $\mathbb{F}_2$ using the classes in $\mathbb{F}_1 \# \CPb = \mathbb{CP}^2 \# 2\CPb$. It is well known that in a Hirzebruch surface $\mathbb{F}_n$ there are two types of elementary transformations: one transforms $\mathbb{F}_n$ to $\mathbb{F}_{n+1}$ for all positive $n$, and the other transforms $\mathbb{F}_n$ to $\mathbb{F}_{n-1}$ for all positive $n$ (see for example the proof of Theorem 2 in \cite{Ul}, page 272-273). The first transformation  blows up a point $P$ on the exceptional sphere section, and then blows down the proper transform of the fiber passing through $P$. The second transformation blows up a point $Q$ on one of the fibers $F$, outside the exceptional section, and then blows down the proper transform of the fiber $F$. 

\end{remark}

Let us illustrate the above ideas for $\mathbb{F}_2$ with two specific examples below.

\begin{example} Let us consider a smooth quadric curve $C$ given by the class $2h$ and a line $L$ given by the class $h$ in $\mathbb{CP}^2$. The curves $C$ and $L$ are depicted as the blue and the purple curve in each of the Figures~\ref{cl1} and ~\ref{cl2}. Note that either $L$ intersects $C$ generically at two distinct points $p$ and $q$ as shown in Figure~\ref{cl1} or at one point $p$ with multiplicity $2$ as in Figure~\ref{cl2}. 

In Figure~\ref{cl1}, we consider two pencils of lines in $\mathbb{CP}^2$ passing through the points $p$ and $q$, respectively. Let us blow up the base points $p$ and $q$ of these two pencils. By blowing down the $-1$ curve $h-e_1-e_2$ in $\mathbb{CP}^2 \# 2\CPb$, we obtain $\mathbb{F}_0$. Note that $h-e_1$ and $h-e_2$ descend to the classes of $\mathbb{S}^2$ fibers in $\mathbb{F}_0$, and $2h - e_1 - e_2$ descends to the class of a section to both fibrations. If we blow up a point on the exceptional sphere $e_1$, instead of the second base point $q$, we can obtain $\mathbb{F}_2$ by blowing down the resulting $-1$ curve $h-e_1-e_2$ in $\mathbb{CP}^2 \# 2\CPb$.  

\begin{figure}{h}
\begin{center}
\scalebox{0.58}{\includegraphics{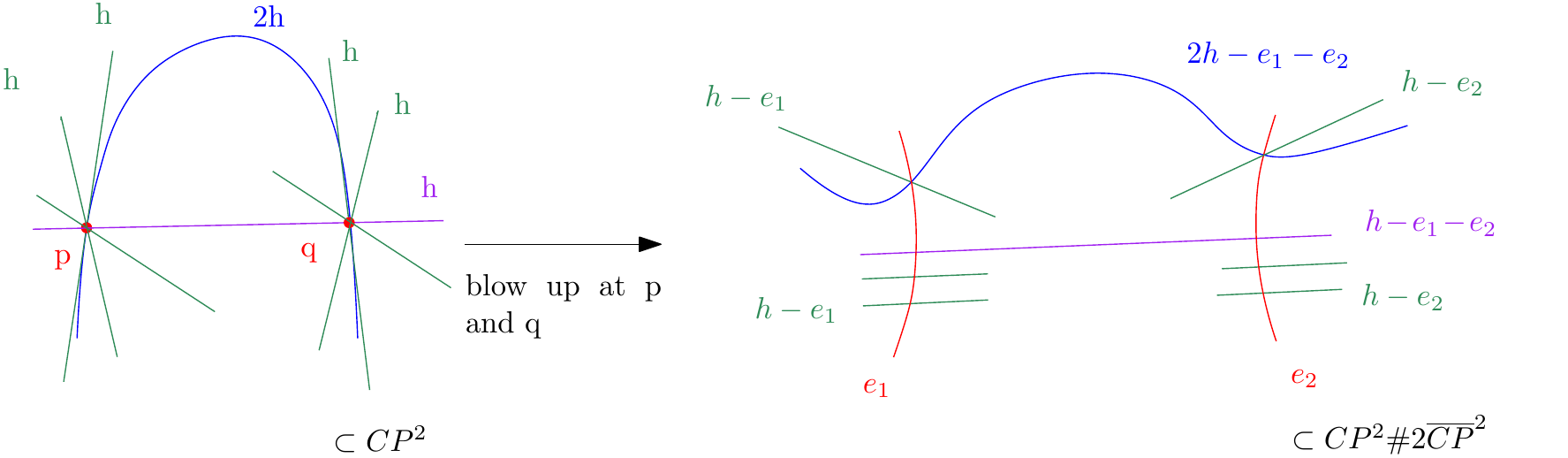}}
\caption{}
\label{cl1}
\end{center}
\end{figure}

In Figure~\ref{cl2}, we consider a pencil of lines in $\mathbb{CP}^2$ passing through the point $p$, and let $L$ (the purple line) be a tangent line to the curve $C$ (the blue curve) at point $p$ with multiplicity $2$. Let us blow up the base point $p$ of the pencil and point $q$ on the exceptional sphere $e_1$ as shown in the Figure~\ref{cl2}. By blowing down the $-1$ curve $h-e_1-e_2$ in $\mathbb{CP}^2 \# 2\CPb$, we obtain $\mathbb{F}_2$. Note that $h-e_1$ descends to the class of \ $\mathbb{S}^2$ fiber in $\mathbb{F}_2$ and $2h - e_1 - e_2$ descends to the class of $+2$ sphere section in $\mathbb{F}_2$.

\begin{figure}{h}
\begin{center}
\scalebox{0.58}{\includegraphics{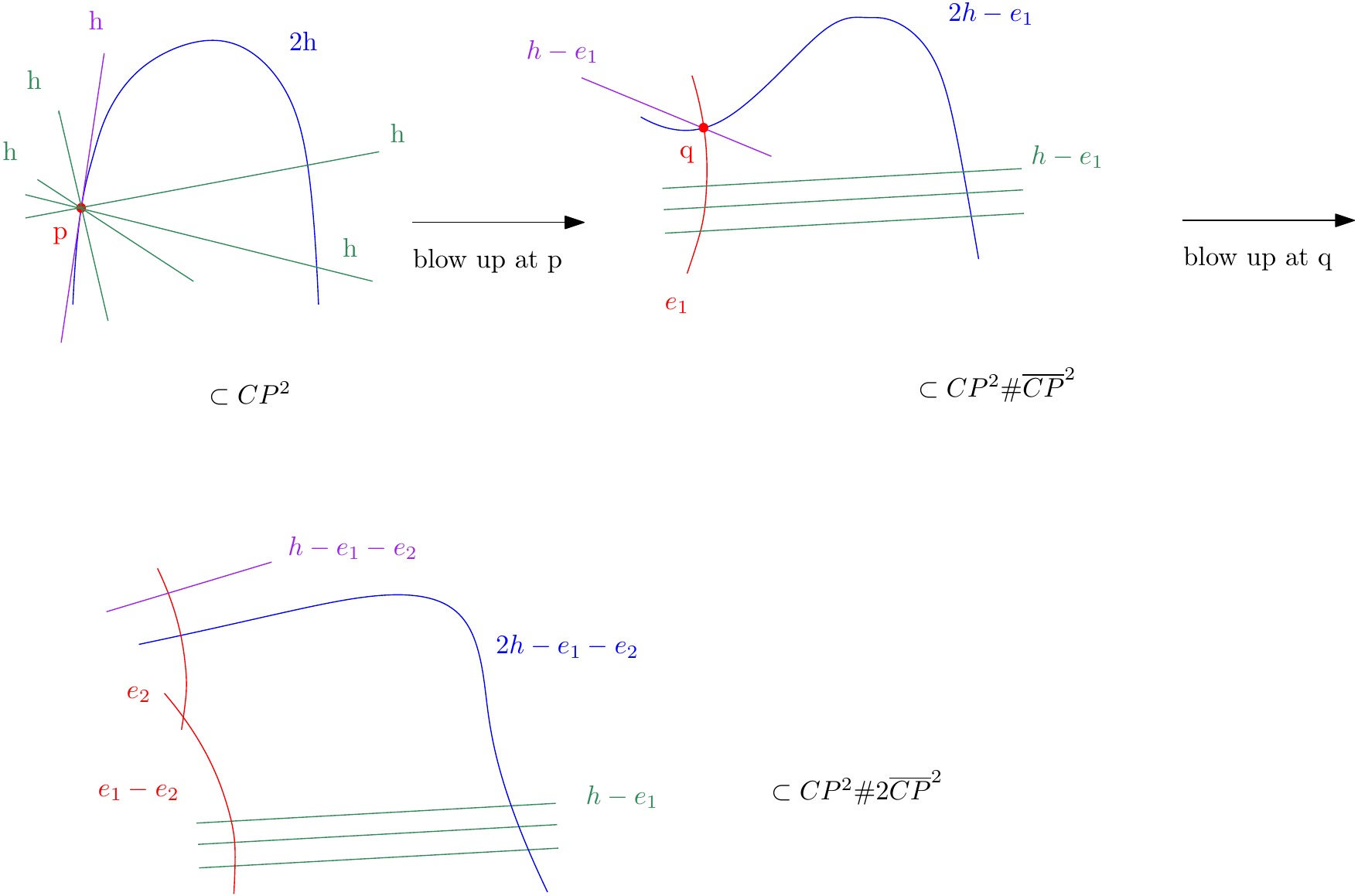}}
\caption{}
\label{cl2}
\end{center}
\end{figure} 

\end{example}

In the sequel, we will also use $\mathbb{F}_3 \cong \mathbb{CP}^2 \# \CPb$, for which a similar analysis of classes could be carried out using Remark~.\ref{crm}. We will take the classes $F, C_{\infty}, C_0$, where $F^2=0, C_{\infty}^2 = 3, C_0^2 = -3$ and $C_{\infty} \cdot C_0 = 0, C_{\infty} \cdot F =1, C_0 \cdot F =1$, and the classes $h, e_1$ of squares $1$ and $-1$ in $\mathbb{CP}^2 \# \CPb$. Similarly as in $\mathbb{F}_2$ case, we can assume that
\begin{equation}
F=h-e_1, \;\; C_{\infty} = 2h-e_1, \;\;C_0 = 2e_1-h.
\end{equation}  

Next, we discuss some algebraic properties of the Hirzebruch surface and state a few theorems and propositions which will be very useful for our purposes. These results are helpful for understanding the fibration structures on blow ups of the Hirzebruch surface. We also mention some needed results from \cite{Gong2}, where the classification and Mordell-Weil groups on these fibrations are thoroughly studied. The proofs of these propositions can be found in \cite{Hart, Harts} (see Corollary 2.18 in \cite{Hart}, page 380, Lemma 3.1 in \cite{SaSa}, page 865), Theorems 2.2, 3.1, Lemmas 3.2, pages 3-8 in \cite{Kit1}, and Theorem 1.6, Corollary 2.2 in \cite{Gong2}. In order to avoid any confusion, we will follow the notation of \cite{Hart, SaSa, Kit1} below.

Let $\pi: \mathbb{F}_{n} \rightarrow \mathbb{CP}^{1}$ be the Hirzebruch surface of degree $n$ with $0 \leq n \leq g$, where $g$ is the genus of a regular fiber. Notice that the Picard group $Pic(\mathbb{F}_{n})$ or $NS(\mathbb{F}_{n})$ is generated by the classes of the infinity section $C_{\infty}$ and fiber $F$ of $\pi$. The intersection pairings on $NS(\mathbb{F}_{n})$ are determined by ${C_{\infty}}^{2} = n$, $C_{\infty} \cdot F=1$, and $F^{2} = 0$. Also, the zero section $C_{0}$ of $\mathbb{F}_{n}$ is equal to $C_{\infty} - nF$, with ${C_{0}}^{2}=-n$.

The next Proposition will be used in the proof of Proposition \ref{prop2}, but it is also relevant to our discussion earlier in this subsection.

\begin{prop}\label{prop1}\cite{Hart} Let $D$ be the divisor $aC_{0} + bF$ on the rational ruled surface $\mathbb{F}_{n}$, and $n \geq 0$. Then: 
\begin{enumerate}[label=(\alph*)]
\item $D$ is very ample $\iff$ $D$ is ample $\iff$ $a > 0$ and $b > an$
\item the linear system $|D|$ contains an irreducible nonsingular curve $\iff$ it contains an irreducible curve $\iff$ $a=0$, $b=1$ (namely $F$); $a=1$, $b=0$ (namely $C_{0}$); or $a>0$, $b > an$; or $n > 0$, $a > 0$, $b=an$.  
\end{enumerate}
\end{prop}

The next proposition and the following discussion are helpful for understanding the fibration structure on the blowing up of the Hirzebruch surface $\mathbb{F}_{n}$.

\begin{prop}\label{prop2}\cite{SaSa} Let $k = g+1 - n > 0$. Then 

\begin{enumerate}[label=(\alph*)]

\item the linear system $2C_{\infty} + kF_{0}$ on the rational ruled surface $\mathbb{F}_{n}$ is very ample
\item a general member $D$ of  $|2C_{\infty} + kF_{0}|$ is a non-singular irreducible hyperelliptic curve of genus $g$.
\end{enumerate}

\end{prop}

\begin{proof} Since $k = g+1 - n >0$, the first part follows from the Proposition~\ref{prop1} given above.
Proposition~\ref{prop1} also implies that there exists a nonsingular irreducible member of the linear system $|D|$. Since a natural projection $D \rightarrow {\mathbb{P}}^{1}$ is a $2:1$ map, $D$ is a hyperelliptic curve. Using the canonical class formula $K_{\mathbb{F}_n}= -2C_{\infty}+ (n-2)F$, we compute $g(D) = \frac{(K_{\mathbb{F}_n} + D) \cdot D} {2} + 1 = \frac{(g-1)F \cdot D}{2}+1 = g$. Moreover, using the very ampleness of the linear system $|2C_{\infty}+ kF|$, we see that its generic smooth irreducible members $D_{0}$ and $D_{1}$ determine a Lefschetz pencil on $\mathbb{F}_{n}$. This means that a generic member $\{D_t\}_{t \in {\mathbb{P}^{1}}}$ of the pencil $\{D_t\}_{t \in {\mathbb{P}^{1}}}$, given by $D_{0}$ and $D_{1}$, is smooth and every member in the pencil is irreducible and has at most one node as its singularity. \end{proof}

Next note that $D_{0} \cdot D_{1} =(2C_{\infty}+ kF)^{2} = 4n+ 4k = 4g+4$ and we can assume that $D_0$ and $D_1$ intersects each other transversely. Therefore, we have $4g+4$ distinct points $P_{1}, \cdots , P_{4g+4}$ which are the base points of the pencil.  Furthermore, we can assume that these points do not lie on the section $C_0$ and any two of them are not on the same fiber of $\pi$. Under these assumptions, let $\phi: X \rightarrow \mathbb{F}_{n}$ be the blowing up of the points $P_{1}, \cdots , P_{4g+4}$, then we obtain the fibration $f: X  \rightarrow {\mathbb{P}^{1}}$ of curves of genus $g$. In this paper, we consider the case $g=2$, where the total space $X$ of the given fibration is  $\mathbb{F}_{n} \# 12\CPb = \mathbb{CP}^2 \# 13\CPb$

We will need the following theorems and corollaries derived~\cite{Kit1} in the proof of one of our main result in Section~\ref{MT}. More precisely, we will use Kitagawa's genus two pencil given in Theorems~\ref{main} and Corollary~\ref{lemma1}. For the proofs, we refer the reader to \cite{Kit1}.

\begin{theorem}\label{main1} [\cite{Kit1}] Let $X$ be a rational surface and $f: X \rightarrow \mathbb{CP}^{1}$ a relatively minimal fibration of genus $g \geq 2$. Assume that the Picard number $\rho(X) = 4g+6$. Then there exists a birational morphism $\mu: X \rightarrow \mathbb{F}_{n}$ with $n \leq g+1$ such that the following conditions (i), (ii) hold
\begin{enumerate}[label=(\roman*)]
\item $\mu_{*}F$ is linearly equivalent to $(2 C_{0} + (g+n+1)F)$.
\item The pull-back to $X$ of a $(-1)$-curve contracted by $\mu$ intersects with $F$ at just one point.
\end{enumerate}
In particular, $F$ is a hyperelliptic curve and $f$ has at least one $(-1)$-section.
\end{theorem}


\begin{theorem}\label{main}[\cite{Kit1}] Let $X$ be a rational surface and $f: X \rightarrow \mathbb{CP}^{1}$ a relatively minimal fibration of genus $g \geq 2$. Assume that the Picard number $\rho(X) = 4g+6$. Let $\mathbb{K} = f^{\star}(\mathbb{C}(\mathbb{P}^1))$. Then the following assertions are equivalent.

\begin{enumerate}[label=(\roman*)]

\item Mordell-Weil group of $f$ is trivial.

\item $f$ has a reducible fibre whose dual graph is as in Figure~\ref{figur3}. Here the empty circles in the figure are $-2$ curves, the oval is a $-(g+1)$ curve and the numbers without circles denote the multiplicities of components in the reducible fiber

\item $f: X \rightarrow \mathbb{CP}^{1}$ is obtained from $\mathbb{F}_{g}$ by eliminating the base points of the following pencil $\Lambda$ : Let $C_0$ be the minimal section and $F$ a fibre of $\mathbb{F}_{g}$. Take a curve $H_{[g]}$ which is linearly equivalent to $2 C_0 + (2g + 1)F$ and which is tangent to $F$ at the intersection point of $F$ with $C_0$. Then $2 C_0 + (2g + 1)F$ and $H_{[g]}$ generate the pencil $\Lambda$.

\end{enumerate}
\end{theorem}

\begin{figure}
\begin{center}
\scalebox{0.58}{\includegraphics{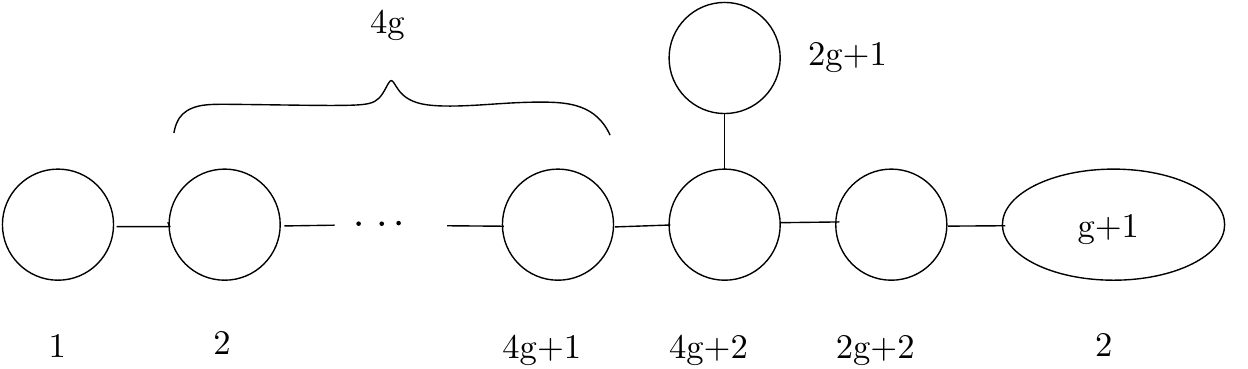}}
\caption{}
\label{figur3}
\end{center}
\end{figure}


\begin{cor}\label{lemma1}[\cite{Kit1}] Assume that $f$ has a reducible fibre whose dual graph is as in Figure~\ref{figur3}.
Then 
\begin{enumerate}[label=(\roman*)]
\item There exists a birational morphism $\mu : X \rightarrow \mathbb{F}_{g}$ such that the pencil $\Lambda$ as in (iii) in Theorem~\ref{main} is obtained from a base-point-free pencil $|F|$ as images by $\mu$.
\item Mordell-Weil group of f is trivial. In particular, a $(-1)$-section of $f$ is unique.
\end{enumerate}
\end{cor}

\begin{cor}\label{lemma2}[\cite{Kit1}]
Assume that $n = g$ or $g + 1$ for all birational morphisms $X \rightarrow \mathbb{F}_{n}$ satisfying conditions (i), (ii) in Theorem~\ref{main}. Then f has a reducible fibre whose dual graph is as in Figure~\ref{figur3}.
\end{cor}

\section{Singular fibres in genus two pencils}\label{singular}

\subsection{Classification of singular fibres in pencils of curves of genus two}

In \cite{Kod}, Kodaira gave the classification of possible singular fibres in pencils of elliptic curves, and showed that in a pencil of elliptic curves, each fibre is either an irreducible curve of arithmetic genus one, i.e. an elliptic curve or a rational curve with a node or a cusp, or a sum of rational curves of self-intersection $-2$ which fall into seven different types. In \cite{Ogg}, Ogg applied Kodaira's argument to pencils of curves of genus two. He classified all possible numerical types of fibres in pencils of genus two curves, and showed that there are $44$ types. Iitaka \cite{Ita} also gave such a classification independently. These fibers were shown actually to arise by Winters (\cite{Wi}). Later Namikawa and Ueno gave geometrical classification of all fibres in pencils of genus two curves \cite{NamU, NamU2}. More precisely, let $\pi: X \rightarrow \mathbb{D}$ be a family of curves of genus two over a disc $\mathbb{D}=\{t \in \mathbb C, |t|<\epsilon\}$, where $X$ is a non-singular (complex analytic) surface free from exceptional curves of the first kind, and $\pi$ is smooth over the punctured disc $D'=\mathbb{D}-\{0\}$. Thus, for every $t \in D'$ the fibre $\pi^{-1}(t)$ is a compact non-singular curve (Riemann surface) of genus two and the restriction of $\pi$ to $D'$ is a topological fibre bundle. For such a family, Namikawa and Ueno defined a multi-valued holomorphic map $T_{\pi}$ from $D’$ into the Siegel upper half plane of degree two, and defined three invariants called ``monodromy", ``modulus point" and ``degree". They showed that the family $\pi$ is completely determined by $T_{\pi}$, and its singular fiber by these three invariants. Hence all types of fibers were classified by these invariants and they listed them up in a table. There are about $120$ types divided into five groups.

The article \cite{Gong} of Gong, Lu and Tan, which uses Namikawa and Ueno's classification, has studied the relatively minimal, isotrivial fibrations of genus $g\geq2$, which will be helpful for our purposes. A fibration is called relatively minimal if no fiber contains an exceptional curve. It is called isotrivial if all smooth fibers are isomorphic to each other. Moreover, a non-trivial family $f:S \rightarrow \mathbb{CP}^1$ of complex curves of genus $g\geq1$ admits at least two singular fibers. Let $f:S \rightarrow \mathbb{CP}^1$ be a fibration with exactly two singular fibers $F_1$ and $F_2$. In this case, $f$ is isotrivial according to \cite{Gong}. From Theorem 1.2 in \cite{Gong} we know that

\begin{theorem} (\cite{Gong}). Let $f:S \rightarrow \mathbb{CP}^1$ be a relatively minimal fibration of genus $g=2$ with two singular fibers $F_1$ and $F_2$. Then $F_1 - F_2$ are one of the following 11 types
I*-I*, II-II, III-III, IV-IV, V-V*, VI-VI, VII-VII*, (VIII-1)-(VIII-4), (VIII-2)-(VIII-3), (IX-1)-(IX-4), (IX-2)-(IX-3) 
where each number denotes a singular fiber in \cite{NamU}.
\end{theorem}

In the above theorem and throughout this paper, we use the enumeration of singular fibers as in \cite{NamU}.

\begin{remark} Each fibration $f: S \rightarrow \mathbb{CP}^1$ above is a relatively minimal model of the normalized double cover $\pi: Y \rightarrow \mathbb{F}_n$ over a Hirzebruch surface $\mathbb{F}_n \rightarrow \mathbb{CP}^1$ branched along a curve $B$ (\cite{Gong}, p.90). 
\end{remark}

In this paper, we will work with the fibrations of types (VIII-1 - VIII-4) and (V - V*). The following proposition-definition is needed in our computation of the total space of each of these fibrations. Our computation will show that the total space of each of these fibrations is $\CP\#13\CPb$.

\begin{definition} (\cite{Gong}, p.85)
Let $f:S \rightarrow C$ be a relatively minimal fibration of genus $g$ over a smooth curve of genus $b$. Then the relative numerical invariants of the fibration are defined as follows

\begin{eqnarray} 
K_f^2 &=& c_1^2(S) - 8(g-1)(b-1)\\
\chi_f &=& \chi(\mathcal{O}_S) - (g-1)(b-1)\\
q_f &=& q(S) - g(C)\\
e_f &=& c_2(S) - 4(g-1)(b-1) = \sum_F (\chi_{top}(F)-(2-2g))
\end{eqnarray}
where the summation is over the singular fibers and $\chi_{top}$ is the topological Euler characteristic.
\end{definition}

The fibrations of types (VIII-1 - VIII-4) and (V - V*) are genus 2 fibrations over $\mathbb{S}^2$ and for each, the following holds: $K_f^2 = 4$, $\chi_f = 2$, $q(S)=0$ (\cite{Gong}, Section 5.1, Group (4), p.90). Hence
\begin{eqnarray} 
K_f^2 = 4 &=& c_1^2(S) + 8 \Rightarrow \nonumber \\
c_1^2(S) &=& -4 \\
\chi_f = 2 &=&  \chi(\mathcal{O}_S) + 1 \Rightarrow \nonumber \\
\chi(\mathcal{O}_S) &=& 1
\end{eqnarray}

Next, using the formulas $e(S) = 12\chi(\mathcal{O}_S) -  c_1^2(S)$ and $\sigma(S) = c_1^2(S) - 8\chi(\mathcal{O}_S)$, we compute $e(S) = 16$ and $\sigma(S) = -12$, and conclude that the total space of each of these $4$ fibrations is $\CP \# 13\CPb$.

\subsection{Genus two pencils in K3 surfaces} Our goal in this subsection is to present the proof of Lemma~\ref{pencils} while presenting some nice constructions of genus two pencils in K3 surfaces. Our discussion closely follows \cite{Kon}. 

In the final section of our paper, we will use these pencils to construct exotic minimal symplectic $4$-manifolds homeomorphic but not diffeomorphic to $3\CP \# k \CPb$ for $k=16, \cdots, 19$. Our discussion and notation closely follow \cite{Kon} and we refer the reader to \cite{Kon} for further details.

Let us consider the diagonal action of $PGL(2, \mathbb{C})$ on $({{\mathbb{CP}}^1})^{5}$. Let $\{P_1,\cdots,P_5\}$ be an ordered stable point in $({{\mathbb{CP}}^1})^{5}$ of this action. Such a point defines a homogeneous polynomial of degree $5$ in two variables, say $x_1$ and $x_2$. We will denote this polynomial by $f_5(x_1, x_2)$. Let $C$ be the plane quintic curve defined by the equation
\begin{equation}
{x_0}^5 = f_5(x_1, x_2) = \prod_{i=1}^{5} (x_1 - \lambda_ix_2)
\end{equation}
 Let us consider the following projective transformation of ${\mathbb{CP}}^{2}$:
\begin{equation}\label{projtran}
 g:(x_0 :x_1 :x_2) \longrightarrow (\zeta x_0 :x_1 :x_2)
\end{equation}
where $\zeta$ is a primitive $5$-th root of unity.

Notice that $g$ acts on $C$ as an automorphism of order $5$. We will denote by $E_0$ and $L_i$ $(1 \leq i \leq 5)$ the lines defined by the equations
\begin{equation}
E_0 :x_0 =0, L_i : x_1 = \lambda_ix_2.
\end{equation}

All $L_i$ are members of the pencil of lines through $(1:0:0)$ and $L_i$ meets $C$ at $(0:\lambda_i :1)$ with multiplicity $5$.

Let $S$ denote the minimal resolution of the double cover of $\CP$ branched along the sextic curve $E_0 + C$. It is easy to verify that $S$ is a $K3$ surface. Let us denote by $\tau$ the covering transformation of $S$. The projective transformation $g$ in (\ref{projtran}) induces an automorphism $\sigma$ of $S$ of order $5$. Let us denote by the same symbol $E_0$ the inverse image of the line $E_0$. 

Let us now consider the following two cases:

\begin{case} 
The equation $f_5 = 0$ has no multiple roots.
\end{case}

In this case the curves $E_i$ ($1 \leq i \leq 5$), which are obtained as exceptional curves of the minimal resolution of the oridinary double point singularities, corresponding to the intersection of $C$ and $E_0$, are all $(-2)$-curves. The inverse image of $L_i$ is the union of two smooth rational curves $F_i$, $G_i$ such that $F_i$ is tangent to $G_i$ at one point, and both $F_i$ and $G_i$ have self-intersection $-3$. Moreover, the following relations hold in the Picard group of $S$:
\begin{equation}
\begin{array}{l}

5E_0 =   \displaystyle\sum_{i=1}^{5}(F_{i} - 2E_{i}), \\
G_{i} + F_{i} = 2E_{0} + \displaystyle\sum_{j\neq 0, i}E_{j}  \\ 
\end{array}
\end{equation}

Let $p$ and $q$ be the inverse images of $(1 : 0 : 0)$. We may assume that all $F_i$ (resp. $G_i$) are passing through $p$ (resp. $q$). $\sigma$ preserves each curve $E_i$, $F_j$, $G_j$ ($0 \leq i \leq 5$, $1 \leq j \leq 5$) and $\tau$ preserves each $E_i$ and $\tau(F_i) = G_i$.


\begin{case} 
$f_5 = 0$ has a multiple root.

\end{case}

In this case the double cover $S$ has a rational double point of type $D_7$. Hence $S$ contains $7$ smooth rational curves $E_j'$ , ($1 \leq j \leq 7$) whose dual graph is of type $D_7$. We assume that $E_1'$ meets $E_0$, and $E_1' \cdot  E_2'=E_2' \cdot E_3'=E_3' \cdot E_4'=E_4' \cdot E_5'=E_5' \cdot E_6'=E_5' \cdot E_7'= 1$. If $\lambda_i$ is a multiple root, then $F_i$ and $G_i$ are disjoint and each of them meets one componet of $D_7$, for example, $F_i$ meets $E_6'$ and $G_i$ meets $E_7'$.

\subsubsection{A pencil of curves of genus two}\label{pencilK3} 
The pencil of lines on $\CP$ through\\ $(1 : 0 : 0)$ gives a pencil of curves of genus two on $S$. Each member of this pencil is invariant under the action of the automorphism $\sigma$ of order $5$. Hence a general member is a smooth curve of genus two with an automorphism of order five. Such a curve is unique up to isomorphism and is given by $y^2 = x(x^5 + 1)$ (see Bolza \cite{Bol}). If $\lambda_i$ is a simple root of the equation $f_5 = 0$, then the line $L_i$ defines a singular member of this pencil consisting of three smooth rational curves $E_i + F_i + G_i$. We call this singular member a singular member of type I. If $\lambda_i$ is a multiple root of $f_5 = 0$, then the line $L_i$ defines a singular fiber consisting of nine smooth rational curves $E_1', \cdots, E_7'$ , $F_i$, $G_i$. We call this a singular fiber of type $II$. The two points $p$ and $q$ are the base points of the pencil. After blowing up at $p$ and $q$, we have a base point free pencil of curves of genus two in $K3\# 2 \CPb$. The singular fibers of such pencils are completly classified by Namikawa and Ueno \cite{NamU}. The type I (resp. type II) corresponds to [IX-2] (resp. [IX-4] ) in \cite{NamU}. Consequently, we have the following lemma, which appeared in \cite{Kon}.

\begin{lemma}\label{pencils}\cite{Kon} The pencil of lines on $\CP$ through $(1 : 0 : 0)$ gives rise to a pencil of curves of genus two on $K3$ surface. A general member of this pencil is a smooth curve of genus two with an automorphism of order five. In case that $f_5 = 0$ has no multiple roots, it has five singular members of type $I$. In case that $f_5 = 0$ has a multiple root (resp. two multiple roots), it has three singular members of type $I$ and one singular member of type $II$ (resp. one of type $I$ and two of type $II$).
\end{lemma}

\section{Nodal spherical deformation of singular fibers of Lefschetz fibrations}\label{Nodal} 

In what follows, we introduce a useful technique called \textit{$g$-nodal spherical deformation} which can be applied to singular fibers of a genus $g\geq 2$ Lefschetz fibration over $\mathbb{S}^2$, and prove a few lemmas that will be used in the course of the proofs of our main theorems. For the sake of convenience, let us first recall some fundamental definitions and facts concerning the Lefschetz fibrations, and list some examples of Lefschetz fibrations for which our nodal spherical defomation technique will be applied.

\begin{definition}\label{LF}\rm
Let $X$ be a closed, oriented smooth $4$-manifold. A smooth map $f : X \rightarrow \mathbb{S}^2$ is a  \textit{genus-$g$ Lefschetz fibration} if it satisfies the following conditions: \\
(i) $f$ has finitely many critical values $b_1,\ldots,b_m \in \mathbb{S}^2$, and $f$ is a smooth $\Sigma_g$-bundle over $\mathbb{S}^2-\{b_1,\ldots,b_m\}$, \\
(ii) for each $i$ $(i=1,\ldots,m)$, there exists a unique critical point $p_i$ in the \textit{singular fiber} $f^{-1}(b_i)$ such that about each $p_i$ and $b_i$ there are local complex coordinate charts agreeing with the orientations of $X$ and $\mathbb{S}^2$, on which $f$ is of the form $f(z_{1},z_{2})=z_{1}^{2}+z_{2}^{2}$, \\
(i\hspace{-.1em}i\hspace{-.1em}i) $f$ is relatively minimal (i.e. no fiber contains a $(-1)$-sphere)
\end{definition}

Each singular fiber of a Lefschetz fibration is an immersed surface with a single transverse self-intersection, and it is obtained by collapsing a simple closed curve (the \textit{vanishing cycle}) in the regular fiber. If the curve is nonseparating, then the singular fiber is called \textit{nonseparating}, otherwise it is called \textit{separating}. Moreover, a singular fiber can be described by its monodromy, i.e, by a right handed Dehn twist along the corresponding vanishing cycle. 

Let us now recall two well-known familes of hyperelliptic Lefschetz fibrations, which will serve as building blocks in our constructions of new Lefschetz fibrations, and for which our nodal defomation technique will be applied. Let $a_1$, $a_2$, $\cdots$, $a_{2g}$, and $a_{2g+1}$ denote the collection of standard simple closed curves in $\Sigma_g$, which is shown as in Figure~\ref{inv}. We will slightly abuse the notation and refer to the right handed Dehn twists $t_{a_i}$ along the curve $a_i$ also with the same letter $a_{i}$. It is well-known that the following two relations hold in the mapping class group $\Gamma_g$:  

\begin{equation}
\begin{array}{l}

H(g) = (a_1a_2 \cdots a_{2g-1}a_{2g}{a_{2g+1}}^2a_{2g}a_{2g-1} \cdots a_2a_1)^2 = 1,  \\
I(g) = (a_1a_2 \cdots a_{2g}a_{2g+1})^{2g+2} = 1,  \\ 
\end{array}
\end{equation}

\begin{figure}[htb]
\begin{center}
\scalebox{0.58}{\includegraphics{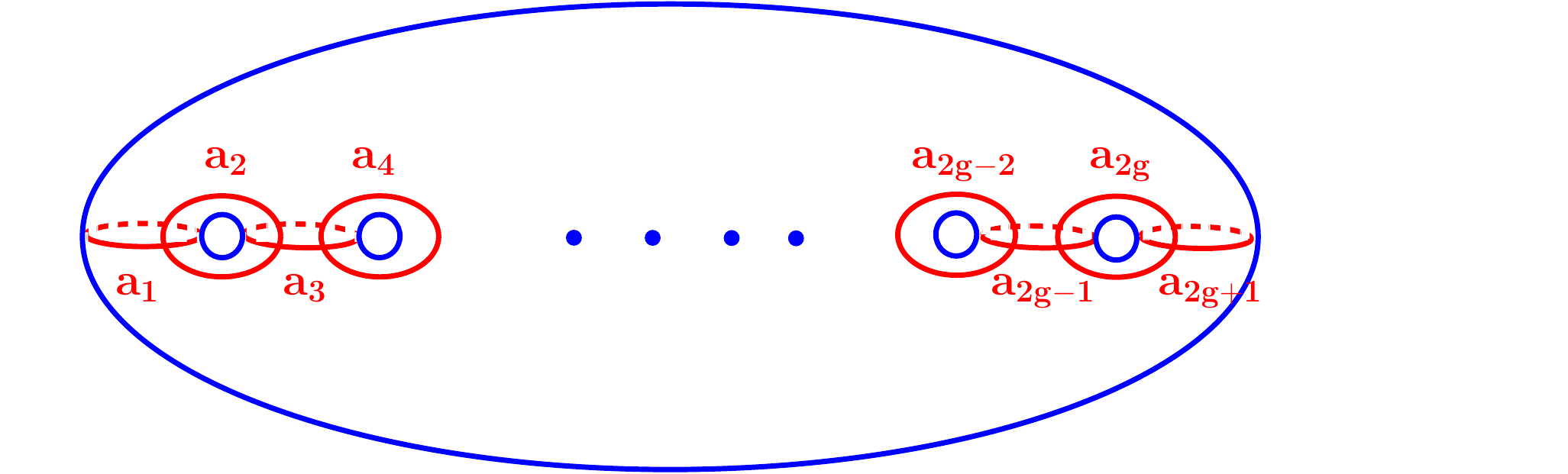}}
\caption{}
\label{inv}
\end{center}
\end{figure}

Let $X(g)$ and $Y(g)$ denote the total spaces of the above two genus $g$ hyperelliptic Lefschetz fibrations over $\mathbb{S}^2$ given by monodromies $H(g) = 1$, and $I(g) = 1$, respectively, in the mapping class group $\Gamma_g$. The first monodromy relation corresponds to genus $g$ Lefschetz fibrations over $\mathbb{S}^2$ with the total space $X(g) = \CP\#(4g+5)\CPb$, the complex projective plane blown up at $4g+5$ points. In the case of the second relation, the total spaces of corresponding genus $g$ Lefschetz fibrations over $\mathbb{S}^2$ are also well-known families of complex surfaces. For example, $Y(2) = K3\#2\CPb$, and the Lefschetz fibration structure given  arises from a well-known pencil in the $K3$ surface with two base points (see for example references \cite{F2, A2}). 

\begin{lemma} Let $f : X \rightarrow \mathbb{D}^2$ denote a Lefschetz fibration with $k$ singular fibers and the monodromy $W = D_{\gamma_1}D_{\gamma_2} \cdots D_{\gamma_k}$ in $\Gamma_g$. Assume that $k \geq g$ and the word $W$ contains a subword $W'$, which consists of a product of $g$ Dehn twists along disjoint nonseparating vanishing cycles. Then the fibration can be deformed so that it contains a spherical $g$-nodal singular fiber.
\end{lemma}

\begin{proof} Using the word $W= \cdots W' \cdots $ and deforming $g$ homologically essential curves corresponding to the nonseparating vanishing cycles on the genus $g$ surface corresponding to the subword $W'$, we obtain a spherical $g$-nodal singular fiber. 
\end{proof}

\begin{figure}
\begin{center}
\scalebox{0.58}{\includegraphics{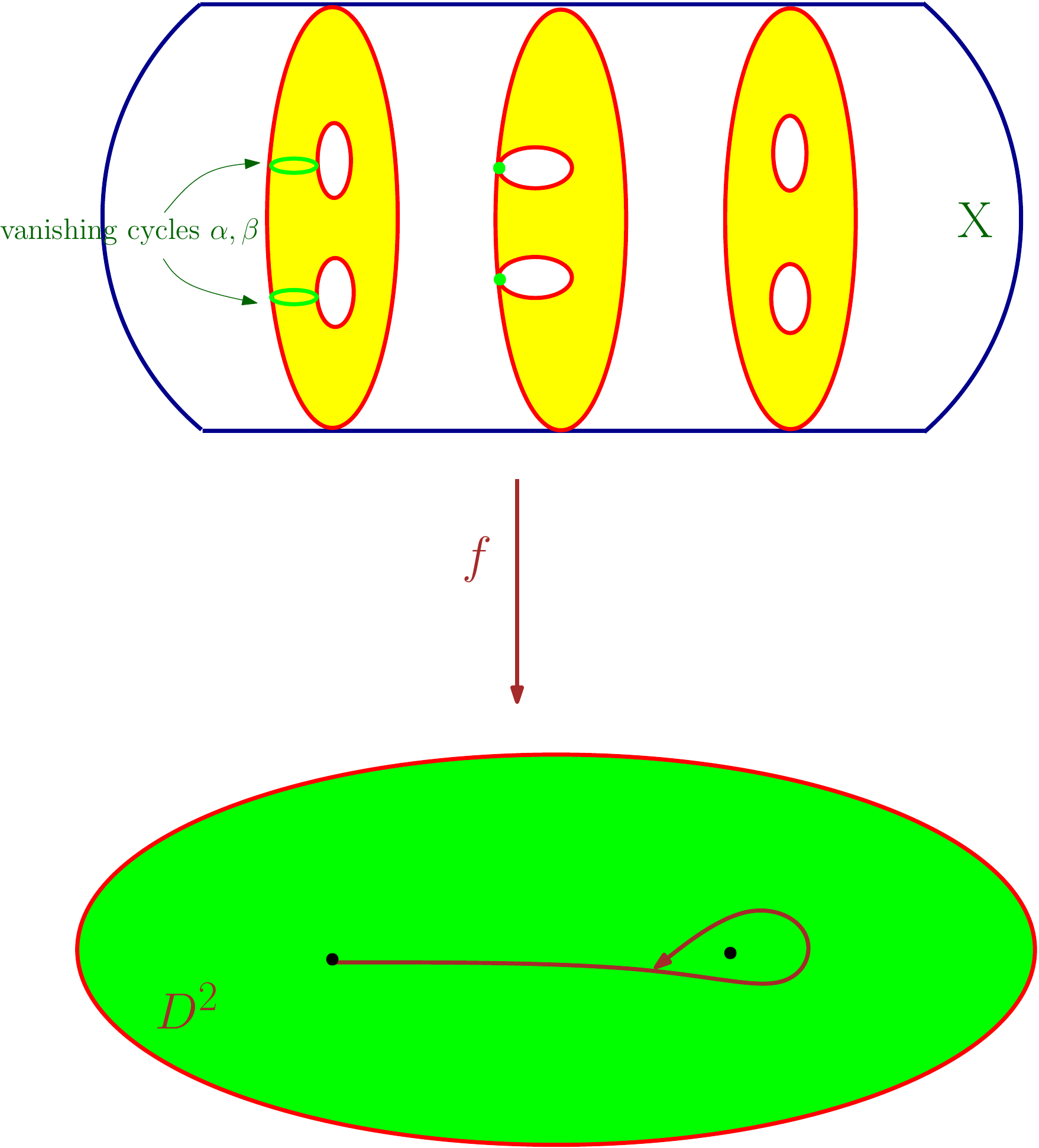}}
\caption{}
\label{dttttt}
\end{center}
\end{figure}

We will now show that the above lemma is applicable to Lefschetz fibration $H(g) = (a_1a_2 \cdots a_{2g-1}a_{2g}{a_{2g+1}}^2a_{2g}a_{2g-1} \cdots a_2a_1)^2 = 1$, and $I(g)= \\
(a_1a_2 \cdots a_{2g}a_{2g+1})^{2g+2} = 1$, when $g = 2$ and using the monodromies $W_1 = a_1a_2a_3a_4$ and $W_2 = a_1a_2a_3a_4a_5$ respectively. We also discuss the cases of $g \geq 3$, but their applications will be studied in a separate paper. 

Before stating the next lemma, let us recall a well-known fact that the conjugate of a Dehn twist is again a Dehn twist: if $f: \Sigma_{g} \rightarrow \Sigma_{g}$ is an orientation-preserving diffeomorphism, then $f \circ t_\alpha \circ f^{-1} = t_{f(\alpha)}$. We will make use of it repeatedly in our computation below.

\begin{lemma}\label{f1} Let $f_{1} : X \rightarrow \mathbb{D}^2$ denote a Lefschetz fibration given by the monodromy $(a_1a_2a_3a_4)$ in $\Gamma_2$. Then it can be deformed to contain two disjoint spherical $2$-nodal singular fibers given by the word below

\begin{equation}
(a_1a_2a_3a_4) = (a_4^{-1}a_1a_3a_4)(a_4^{-1}a_3^{-1}a_2a_4a_3a_4)
\end{equation}  
\end{lemma}

\begin{proof}
At first, for a word $a_1 a_2 \cdots a_n$ in the $\Gamma_g$, by Hurwitz moves we mean either one of the following two equalities:
\begin{eqnarray}
a_1a_2 \cdots a_{i}a_{i+1}\cdots a_n &=& a_1a_2\cdots(a_{i}a_{i+1}a_{i}^{-1})(a_{i})\cdots a_n\\
a_1a_2 \cdots a_{i}a_{i+1}\cdots a_n &=&a_1a_2\cdots(a_{i+1})(a_{i+1}^{-1}a_{i}a_{i+1}) \cdots 
a_n. 
\end{eqnarray}
By applying these moves, the braid relation $a_ia_{i+1}a_i= a_{i+1}a_ia_{i+1}$, and the commutativity relation of disjoint curves, we compute
\begin{eqnarray*}
a_1a_2a_3a_4&=&a_1a_2(a_3a_4a_3^{-1})a_3\\
&=& a_1a_2(a_4^{-1}a_3a_4)a_3\\
&=&a_4^{-1}a_1a_2a_3a_4a_3\\
&=&a_4^{-1}a_1a_3(a_3^{-1}a_2a_3)a_4a_3\\
&=&a_4^{-1}a_1a_3(a_2a_3a_2^{-1})a_4a_3\\
&=&a_4^{-1}a_1a_3a_2a_3a_4a_2^{-1}a_3\\ 
&=&a_4^{-1}a_1a_3a_2(a_4)(a_4^{-1}a_3a_4)a_2^{-1}a_3\\
&=&(a_4^{-1}a_1a_3a_4)(a_4^{-1}a_2a_3a_4a_2^{-1}a_3)\\
&=&(a_4^{-1}a_1a_3a_4)(a_4^{-1}a_2a_3a_2^{-1}a_4a_3)\\
&=&(a_4^{-1}a_1a_3a_4)(a_4^{-1}a_3^{-1}a_2a_3a_4a_3)\\
&=&(a_4^{-1}a_1a_3a_4)(a_4^{-1}a_3^{-1}a_2a_4a_3a_4).
\end{eqnarray*}

Geometrically we can view the above process as in Figure \ref{deform1}. The resulting two singular fibers corresponding to $(a_4^{-1}a_1a_3a_4)$ and $(a_4^{-1}a_3^{-1}a_2a_4a_3a_4)$, are two disjoint spherical fibers with $2$ nodes on each.   

\end{proof}

\begin{remark} Alternatively, by introducing a pair  $a_{4}^{-1}a_{4}$, commuting $a_{4}^{-1}$ with $a_{2}$ and then with $a_{1}$ since these cycles are disjoint, and introducing $a_{3}a_{4}a_{4}^{-1}a_{3}^{-1}$ after $a_{1}$ term one can obtain a proof of the above lemma without using the braid relations. 
\begin{eqnarray*}
a_1a_2a_3a_4&=&a_{1}a_{2} a_{4}^{-1}a_{4} a_{3}a_{4}\\
&=& a_{4}^{-1} a_{1}a_{2} a_{4}a_{3}a_{4}\\
&=&a_{4}^{-1}a_{1} a_{3}a_{4}a_{4}^{-1}a_{3}^{-1} a_{2}a_{4}a_{3}a_{4}\\
\end{eqnarray*}
We thank the referee for making this remark.
\end{remark}

\begin{figure}
\begin{center}
\scalebox{0.80}{\includegraphics{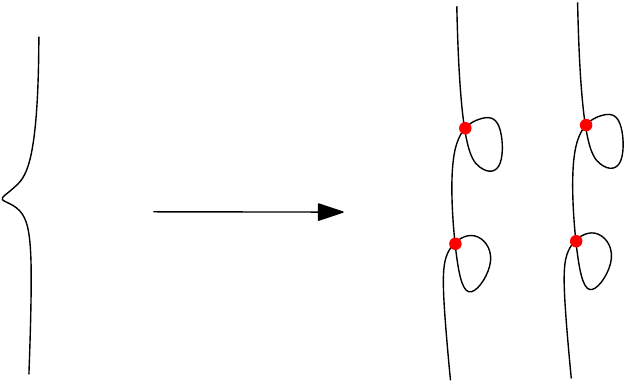}}\caption{Deforming $(2,5)$ cusp into two disjoint $2$-nodal fibers}
\label{deform1}
\end{center}
\end{figure}

\begin{lemma}\label{f2} Let $f_{2} : X \rightarrow \mathbb{D}^2$ denote a Lefschetz fibration given by the monodromy $(a_1a_2a_3a_4a_5)$ in $\Gamma_2$. Then it can be deformed to contain two disjoint spherical $2$-nodal singular fibers given by the word below

\begin{equation}
(a_1a_2a_3a_4a_5) = (a_1a_4)(a_2a_5)a_5^{-1}a_{3}a_4a_3^{-1}a_5.
\end{equation}  
\end{lemma}

\begin{proof} By applying Hurwitz moves, the braid relation $a_ia_{i+1}a_i= a_{i+1}a_ia_{i+1}$, and the commutativity relation for disjoint curves, we compute
\begin{eqnarray*}
a_1a_2a_3a_4a_5&=&a_1a_2(a_4a_4^{-1})a_3a_4a_5\\
&=& a_1a_2a_4(a_3a_4a_3^{-1})a_5\\
&=&a_1a_4a_2a_3a_4a_3^{-1}a_5\\
&=&a_1a_4a_2a_5a_5^{-1}a_{3}a_4a_3^{-1}a_5.
\end{eqnarray*}

Geometrically we may sketch this as in Figure \ref{deform1}. The resulting two singular fibers corresponding to $a_1a_4$ and $a_2a_5$, are two disjoint spherical fibers with $2$ nodes on each, and the third singular fiber corresponding to $a_5^{-1}a_{3}a_4a_3^{-1}a_5$ is the Lefschetz type nodal fiber. \end{proof}

\begin{remark} One can also prove Lemma~\ref{f2}, using Lemma~\ref{f1}. This can be achieved by adding  $a_5$ to the deformed monodromy given in Lemma~\ref{f1}. This obviously will yield to a different monodromy decompositions. We thank the referee for making this observation.
\end{remark}

Now we apply our technique to a genus three Lefschetz fibration.

\begin{lemma}\label{f3} Let $f_{3} : X \rightarrow \mathbb{D}^2$ denote a Lefschetz fibration given by the monodromy 
$(a_1a_2a_3a_4a_5a_6a_7)$ in $\Gamma_3$. Then it can be deformed to contain two disjoint spherical $3$-nodal singular fibers given by the word below
\begin{equation*}
a_1a_2a_3a_4a_5a_6a_7 = 
a_1a_2a_1^{-1} (a_1a_3a_5)  (a_1a_2a_1^{-1})^{-1}(a_1a_2a_1^{-1}) (a_5^{-1}a_4a_5) a_7a_7^{-1}a_6a_7.
\end{equation*}  
\end{lemma}

\begin{proof} By applying Hurwitz moves, we compute
\begin{align*}
a_1a_2a_3a_4a_5a_6a_7= a_1a_2{a_1}^{-1}a_1a_3a_5{a_5}^{-1}a_4a_5a_6a_7 = \\
= [a_1a_2a_1^{-1} (a_1a_3a_5)  (a_1a_2a_1^{-1})^{-1}]  [(a_1a_2a_1^{-1}) (a_5^{-1}a_4a_5) a_7][a_7^{-1}a_6a_7]
\end{align*}

\noindent The resulting two singular fibers corresponding to $[a_1a_2a_1^{-1} (a_1a_3a_5)  (a_1a_2a_1^{-1})^{-1}] $ and $[(a_1a_2a_1^{-1}) (a_5^{-1}a_4a_5) a_7]$ are two disjoint spherical fibers with $3$ nodes on each, and the third singular fiber corresponding to $a_7^{-1}a_6a_7$ is a Lefschetz type nodal fiber.\end{proof}

\begin{lemma} Let $f_{4} : X \rightarrow \mathbb{D}^2$ and $f_{5} : X \rightarrow \mathbb{D}^2$ denote the Lefschetz fibrations given by the relations $(a_1a_2a_3a_4a_5\cdots a_{2g-1}a_{2g})$ and  $(a_1a_2a_3a_4a_5\cdots a_{2g-1}a_{2g} a_{2g+1})$ in $\Gamma_g$. Then they both can be deformed to contain a spherical $g$-nodal singular fiber given by the word below

\begin{equation}
(a_1a_2a_3a_4a_5 \cdots a_{2g-1}a_{2g}) = (a_1a_3a_5 \cdots a_{2g-1}) W''.
\end{equation}  

\begin{equation}
(a_1a_2a_3a_4a_5 \cdots a_{2g-1}a_{2g}a_{2g+1}) = (a_1a_3a_5 \cdots a_{2g-1}) W'''.
\end{equation} 

\end{lemma}

\begin{proof} This is relatively easy to verify by applying Hurwitz moves and the commutativity relation of disjoint curves, and we leave it as an exercise to the reader.
\end{proof}

\section{The Main Theorems}\label{MT}

In this section, we will construct exotic copies of $\CP \# 7\CPb$, $\CP \# 6\CPb$, and  $3\,\CP\#k\,\CPb$ for $k = 16, 17, 18, 19$ starting with certain genus two Lefschetz fibration structures on $\CP\#13\,\CPb$ and $E(2)\#2\,\CPb$, and using the techniques and constructions outlined above in sections \ref{HS}, \ref{singular}, \ref{Nodal}.

\subsection{Fibration VIII-1 - VIII-4}\label{VIII}

In this subsection, we will construct exotic copies of $\CP \# 7\CPb$ and $\CP \# 6\CPb$ using certain genus two fibrations on $\CP \# 13\CPb$ that were studied by Namikawa-Ueno in \cite{NamU} and later by Gong, Lu, Tan in \cite{Gong}. 
\begin{theorem} 
Let $M$ be one of the following 4-manifolds 
\begin{enumerate}
\item $\CP \# 7\CPb$ 
\item $\CP \# 6\CPb$ 
\end{enumerate}
Then there exists an irreducible symplectic 4-manifold homeomorphic but non-diffeomorphic to $M$, obtained from the total space of a genus two fibration with two singular fibers of types VIII-1 and VIII-4, using the combinations of $2$-spherical deformations, symplectic blowups, and the (generalized) rational blowdown surgery.
\end{theorem}

\begin{proof} Our exotic symplectic $4$-manifolds will be obtained from blow ups of the Hirzebruch's surface $\mathbb{F}_2$ via a combination of the deformation, symplectic blow ups and rational blowdown. More specifically, we will construct symplectic embeddings of (generalized) rational blowdown plumbings $C_{11}$ and $C_{23, 11}$, and apply the (generalized) rational blowdown surgery to them.

First, let us recall Theorem~\ref{main} and Corollaries~\ref{lemma1},~\ref{lemma2}, which are derived from it, on the existence of a certain genus two pencil that yields to a fibration with the singularity types VIII-4 and VIII-1. Using this genus two pencil and the sequence of blowups, one can write down the explicit classes corresponding to the components of singular fibers, which will be needed in the computation of Seiberg-Witten invariants. The discussion that follows provides these details.

Let us take a reducible algebraic curve $A$ consisting of the fiber $F=h-e_1$ with multiplicity $5$ and the $-2$ section $C_0 = e_1-e_2$ with multiplicity $2$ in $\mathbb{F}_2$. Let the curve $A$ be defined by a homogeneous polynomial $p_1$. We also take an irreducible algebraic curve $B= 2C_0+5F$ defined by a homogeneous polynomial $p_2$. Note that $B^2= 12, F \cdot B = 2, F \cdot C_0 = 1$. 

According to Theorem~\ref{main}, we may represent $A$ and $B$ curves as at the beginning of Figure \ref{eight}. To get the above given pencil in $\mathbb{F}_2$, which looks complicated at first, the contraction of $-1$ spheres can be used in Corollary~\ref{lemma2}: by blowing down the $-1$ sphere section and the resulting $-1$ curves consecutively, the above given pencil arises in $\mathbb{F}_2$.




The curves $A$ and $B$ represent the same class in homology, and define a Lefschetz pencil $C_{[t_1:t_2]}:= {t_1p_1 + t_2p_2 = 0}$ whose base locus is the point $p$.
\begin{figure}
\begin{center}
\scalebox{0.75}{\includegraphics{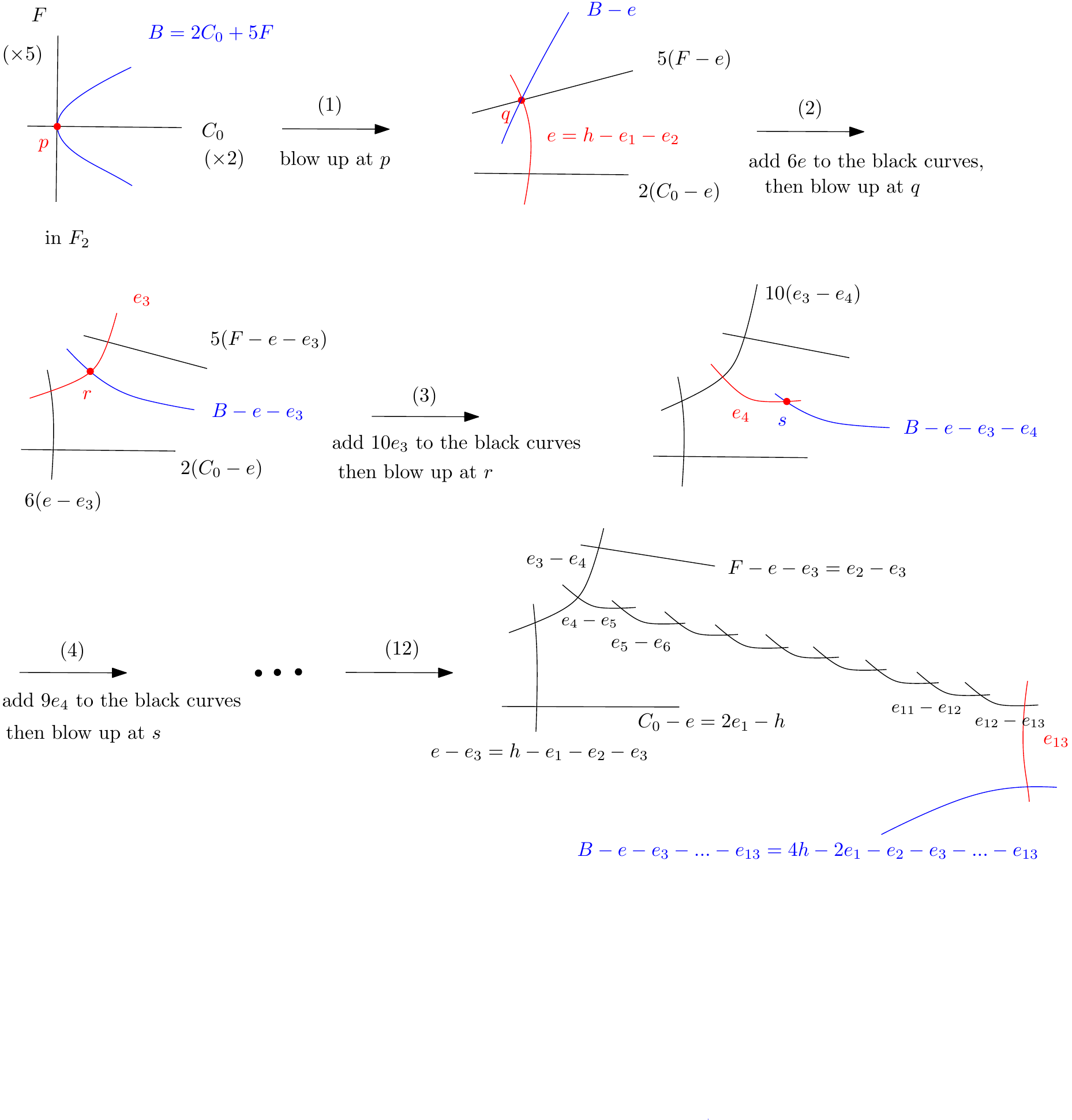}}\caption{Fibers VIII-4, VIII-1}
\label{eight}
\end{center}
\end{figure} 
We blow up the given pencil at the base point $p$ (see Figure \ref{eight}, step 1), and the resulting exceptional divisor is $e=h-e_1-e_2$. To obtain the desired singular fiber of type  VIII-4, we need to redefine the curves $A$ and $B$ as follows: We take $A'$ as the proper transform of curve $A$ together with the divisor $e$ with multiplicity six, and the curve $B'$ as $B-e$, the proper transform of $B$. Notice that now both $A'$ and $B'$ curves represent the homology class $2C_0+5F-e = 4h-2e_1-e_2$, so they define a Lefschetz pencil of genus $2$. Now as $e$ is a part of $A'$ curve and intersects $B-e$ curve at the point $q$, the base locus is nonempty. We blow up at $q$ and denote the resulting exceptional divisor by $e_3$ (see Figure \ref{eight}, step 2). As above we will reset our curves, by adding $e_3$ with multiplicity ten to the proper transform of $A'$ (which is colored black in Figure~\ref{eight}) and as the second curve we take $B-e-e_3$ (the blue curve). Both of the new curves represent the same homology class $2C_0+5F -e -e_3$, and thus define a Lefschetz pencil. The intersection point $r$ of $A'$ and $B'$ curves becomes the base locus at which we blow up and call the exceptional divisor $e_4$ (Figure \ref{eight}). We continue in the same fashion. To equate the homology classes of the black and the blue curves, we add $e_4$ with multiplicity nine to the black curve and blow up its intersection point $s$ with the blue curve. We call the exceptional divisor $e_5$. Then add $e_5$ with multiplicity eight to the black curve and blow up. Notice that the multiplicities decrease by one at each step, so after the 12th blow up, we add $e_{12}$ with multiplicity one to the black curve which intersects the blue curve at one point. Lastly, we blow up at that point and call the divisor that separates the black and the blue curves by $e_{13}$. Note that at this step, the total homology class of the black curve and the homology class of the blue curve $\tilde B$ are both $B-e-e_3-\cdots-e_{13} = 2C_0+5F-e-e_3-\cdots-e_{13}$. Since the homology classes are equal, we do not include the exceptional divisor $e_{13}$ to any component of the Lefschetz pencil and stop at this step (see the last part of Figure \ref{eight} where we ignored the multiplicities of the irreducible components of the black curve). This configuration is symplectically embedded in $\mathbb{F}_2 \# 12\CPb \cong \CP \# 13\CPb$. Moreover, we note that the self intersection of the blue curve is zero; $(\tilde B)^2= (B-e-e_3-\cdots-e_{13})^2= (4h-2e_1 -e_2 -e_3 -\cdots -e_{13})^2 = 0$ so $(\tilde B)^2$ is the generic genus two fiber. Observe that the black curve on the last part of Figure \ref{eight} is the fiber VIII-4, which is the complement of the $(2,5)$ cusp singular fiber VIII-1 \cite{Gong}.

\subsection{Case 1}
Using the singular fibers of types VIII-4 and VIII-1 in $\mathbb{F}_2 \# 12\CPb$ $\cong \CP \#13\CPb$, we will apply our deformation technique and the rational blowdown surgery to construct an exotic copy of $\CP \# 7\CPb$. We first consider the complement of the fiber VIII-4 which is a $(2,5)$ cusp fiber VIII-1. Its monodromy is given by the word $W_{1} = (a_1a_2a_3a_4)$ in the Mapping Class Group $\Gamma_{2}$ \cite{GS} and by Lemma~\ref{f1} this fibration can be deformed to contain two disjoint $2$-nodal spherical singular fibers.

Geometrically we may sketch this as in Figure \ref{cusp}. The resulting two nodal spherical fibers, with $2$ nodes on each, are in $\CP \# 13\CPb$. We blow up these four nodes of the two $2$-nodal fibers, as in Figure \ref{cusp} and resolve their proper transforms $\tilde B-2e_{14}-2e_{15}$ and $\tilde B-2e_{16}-2e_{17}$ with the sphere section $e_{13}$. The resulting sphere $s$ has the following homology class
\begin{eqnarray*}
s&:=& (\tilde B-2e_{14}-2e_{15})+(\tilde B-2e_{16}-2e_{17})+e_{13}\\ 
&=& 8h - 4e_1-2e_2 -2e_3-\cdots-2e_{12}-e_{13}-2e_{14}-\cdots-2e_{17}, 
\end{eqnarray*}
where $e_i$'s are the exceptional divisors and $\tilde B$ is as above. We note that $s^2 = -13$.
\begin{figure}
\begin{center}
\scalebox{0.70}{\includegraphics{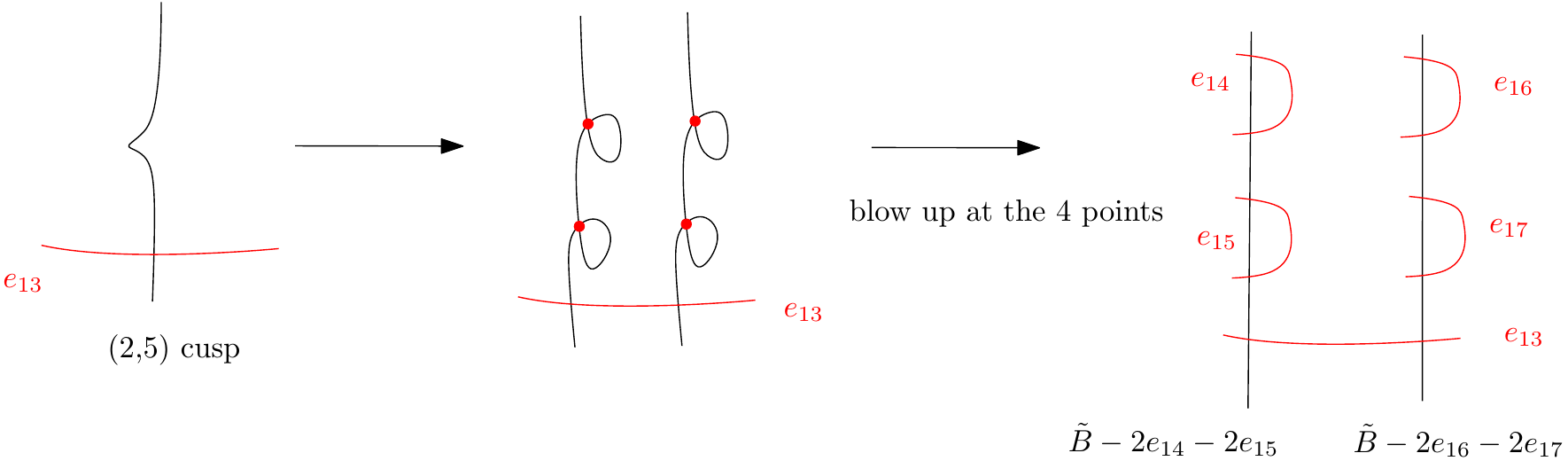}}\caption{Deformation of the $(2,5)$ cusp and the class $s$}
\label{cusp}
\end{center}
\end{figure} 
Also, we have the singular fiber VIII-4 intersecting the sphere $s$ once in $\CP \# 17\CPb$. Consequently, we obtain a plumbing $P$ of length ten as in Figure \ref{bdown} which is symplectically embedded in $\CP \# 17\CPb$. In the plumbing $P$, the homology class of leading $-13$ sphere is given by $s$ above, and we denote the $-2$ spheres of the plumbing $P$ by $u_2,\cdots,u_{10}$ respectively. We rationally blow down this plumbing and call the resulting symplectic manifold $\mathcal Y$ \cite{Sym}. Next, we will show that $\mathcal Y$ is an exotic copy of $\CP \# 7\CPb$.
\begin{figure}
\begin{center}
\scalebox{0.70}{\includegraphics{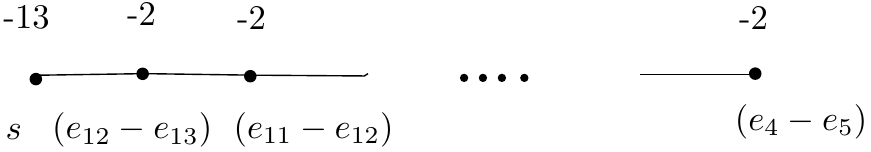}}\caption{Plumbing of length 10}
\label{bdown}
\end{center}
\end{figure} 

Let us first show that $\mathcal Y$ is homeomorphic to $\CP \# 7\CPb$. This is an application of Lemma~\ref{rbd} and Freedman's classification theorem. For the reader's convenience, we spell out the details below. 

Let $\mathcal Y = (\CP \# 17\CPb - P) \cup B$, where $B$ is the rational homology ball whose boundary is the lens space  $L(121, 10)$ which also bounds $P$. We first contract the generator of $\pi_1(\partial P)$ along the sphere $(e_3 - e_4)$, which tells us that  $\CP \# 17\CPb - P$ is simply connected. Note that $(e_3 - e_4)$ intersects $P$ but it is not used in the rational blow down surgery. On the other hand we have the surjection $\pi_1(\partial B) \twoheadrightarrow \pi_1(B)$. Thus, $\mathcal Y$ is simply connected by Van Kampen's theorem. Applying well known formulas, we compute
\begin{eqnarray*}
e (\mathcal Y)&=& e((\CP \# 17\CPb) - e(P) + e(B)\\
&=& 20-11+1\\
&=&10,\\
\sigma (\mathcal Y)&=& \sigma (\CP \# 17\CPb) - \sigma (P) + \sigma (B)\\
&=& -16-(-10)\\
&=&-6.
\end{eqnarray*}
Hence, by Freedman's classification the above follows.

Next, we will show that $\mathcal Y$ is not diffeomorphic to $\CP \# 7\CPb$. First, we know that for every $k > 0$, the manifold $\CP \# k \CPb$ admits a symplectic structure whose cohomology class is given by $w= ah-b_1e_1-\cdots-b_ke_k$ for some positive rational numbers $a,b_1,\cdots,b_k$ with $a>b_1>\cdots>b_k$ and $a>b_1+\cdots+b_k$ (\cite{KS}, Lemma 5.4). Let 
\begin{equation}
w= ah-b_1e_1-\cdots-b_{17}e_{17} 
\end{equation}
be the cohomology class of a symplectic strucutre on $\CP \# 17 \CPb$ with $a>b_1>\cdots>b_{17}>0$ and $a>b_1+\cdots+b_{17}$. Let $K$ be the canonical class of $\CP \# 17 \CPb$, so we have 
\begin{equation}
K= -3h+e_1+\cdots+e_{17}. 
\end{equation}
By direct computation, we see that $K$ is disjoint from all $-2$ spheres $u_2,\cdots,u_{10}$ of the plumbing $P$ in Figure \ref{bdown}. Let $\gamma_1,\cdots,\gamma_{10}$ be the basis of $H^2(P,\mathbb{Q})$ which is dual to $s,u_2,\cdots,u_{10}$. Then from the adjunction formula
\begin{equation*}
\begin{aligned}
K|_P =& (K \cdot s) \gamma_1+(K \cdot u_2) \gamma_2 +\cdots+(K \cdot u_{10}) \gamma_{10}\\
=& (-3h+e_1+\cdots+e_{17})\cdot \\
&(8h - 4e_1-2e_2 -2e_3-\cdots-2e_{12}-e_{13}-2e_{14}-\cdots-2e_{17})\gamma_1\\
=& 11 \gamma_1.
\end{aligned}
\end{equation*}
We calculate the restriction of the symplectic class $w$ on $P$
\begin{equation*}
\begin{aligned}
w|_P &= (w \cdot s) \gamma_1+(w \cdot u_2) \gamma_2 +\cdots+(w \cdot u_{10}) \gamma_{10}\\
&= (ah-b_1e_1-\cdots-b_{17}e_{17})\cdot\\
&(8h - 4e_1-2e_2 -2e_3-\cdots-2e_{12}-e_{13}-2e_{14}-\cdots-2e_{17})\gamma_1\\
&+ (ah-b_1e_1-\cdots-b_{17}e_{17})\cdot (e_{12}-e_{13})\gamma_2+\cdots\\
&+(ah-b_1e_1-\cdots-b_{17}e_{17})\cdot (e_{4}-e_{5})\gamma_{10}\\
&=(8a-4b_1-2b_2-2b_3-\cdots-2b_{12}-b_{13}-2b_{14}-\cdots-2b_{17})\gamma_1\\
&+ (b_{12}-b_{13})\gamma_2 + (b_{11}-b_{12})\gamma_3+ (b_{10}-b_{11})\gamma_4 +(b_{9}-b_{10})\gamma_5\\
&+ (b_8-b_9)\gamma_6+ (b_7-b_8)\gamma_7+(b_6-b_7)\gamma_8+(b_5-b_6)\gamma_9+(b_4-b_5)\gamma_{10}.    
\end{aligned}
\end{equation*}
Let $M$ be the intersection matrix for the plumbing $P$:
$$\quad
M=
\begin{bmatrix} 
-13 & 1 & 0 &0 &0 &0 &0 &0 &0 &0 \\
1    &-2 &1 & 0 &0 &0 &0 &0 &0 &0\\
0&1&-2 &1 & 0 &0 &0 &0 &0 &0\\
0&0&1&-2 &1 & 0 &0 &0 &0 &0\\
0&0&0&1&-2 &1 & 0 &0 &0 &0\\
0&0&0&0&1&-2 &1 & 0 &0 &0\\
0&0&0&0&0&1&-2 &1 & 0 &0\\
0&0&0&0&0&0&1&-2 &1 & 0 \\
0&0&0&0&0&0&0&1&-2 &1\\
0&0&0&0&0&0&0&0&1&-2 
\end{bmatrix}$$
Then, the first column $[-1/121 (10,9,8,7,6,5,4,3,2,1)]$ of $M^{-1}$ gives us $\gamma_1 \cdot \gamma_i$, $i=1,\cdots,10$.  A direct, but lengthy computation shows that \begin{equation*}
\begin{aligned}
K|_P \cdot w|_P &= -11/121 (80a-40b_1-20(b_2+b_3+b_{14}+b_{15}+b_{16}+b_{17})\\
& -19(b_4+b_5+b_6+\cdots
+b_{13})).
\end{aligned}
\end{equation*}
Finally, we compute
\begin{eqnarray*}
K|_{\mathcal Y} \cdot w|_{\mathcal Y} &=& K \cdot w - K|_P \cdot w|_P\\
&=& (-3a+b_1+b_2+\cdots+b_{17})\\
&+&11/121 (80a-40b_1-20(b_2+b_3+b_{14}+b_{15}+b_{16}+b_{17})\\ 
&-&19(b_4+b_5+b_6+\cdots+b_{13}))\\
&=& 1/121 (517a-319b_1-88(b_4+\cdots+b_{13})\\
&-&99(b_2+b_3+\cdots+b_{17}))\\
&>&0.
\end{eqnarray*}
This shows that $\mathcal Y$ is not diffeomorphic to $\CP \# 7\CPb$. In fact, the standard symplectic form on $\CP \# k\CPb$ satisfies $K\cdot w <0$. Also, there is a unique symplectic structure on $\CP \# k\CPb$ for $2 \leq k \leq 9$ up to diffeomorphism and deformation (\cite{LL}, Theorem D). Hence $\CP \# 7\CPb$ does not admit a symplectic structure with $K \cdot w >0$. 

Furthermore, using the adjunction formula and inequalities, and the methods of the article \cite{{OzsSz}}, we have verified the minimality of $\mathcal Y$.  


\subsection{Case 2}
In this subsection, we will construct an exotic copy of $\CP \# 6\CPb$. We will use a generalized rational blowdown plumbing $C_{23,11}$ of the form as shown in Figure \ref{plumbing} for $m=3$ and $k = 9$. 

\begin{figure}
\begin{center}
\scalebox{0.95}{\includegraphics{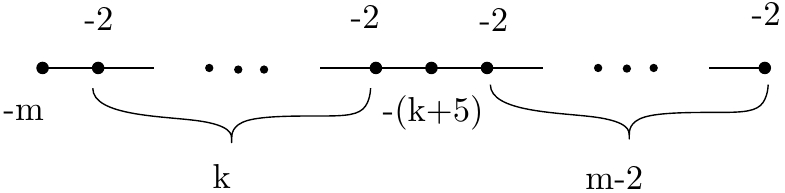}}
\caption{Plumbing for a generalized rational blow down}
\label{plumbing}
\end{center}
\end{figure} 

First, the curve $\tilde B-2e_{14}-2e_{15}$ intersects the exceptional divisor $e_{14}$ twice as in Figure \ref{cusp} above. We blow up one of their intersection points, call the exceptional sphere $e_{18}$. We note that, after the blow up, the $-9$ curve $\tilde B-2e_{14}-2e_{15}-e_{18}$ intersects the proper transform $e_{14}-e_{18}$ of $e_{14}$ once. Let us take the symplectic resolution of the three curves:
\begin{eqnarray*}
s'&:=& (\tilde B-2e_{14}-2e_{15}-e_{18})+(\tilde B-2e_{16}-2e_{17})+e_{13}\\ 
&=& 8h - 4e_1-2e_2 -2e_3-\cdots-2e_{12}-e_{13}-2e_{14}-\cdots-2e_{17}-e_{18}, 
\end{eqnarray*}
where $e_i$'s are the exceptional divisors and $\tilde B$ is as above. We have $s'^2 = -14$ and it intersects $e_{14}-e_{18}$ and also $e_{12}-e_{13}$. Moreover, in Figure \ref{eight}, we symplectically resolve three curves: $(2e_1 -h) + (h-e_1-e_2-e_3)+(e_3-e_4)= e_1-e_2-e_4$ which is a -3 curve. Hence, we obtain a plumbing $P'$ of length $12$ for the rational blow down (see Figure \ref{6}). It can be obtained as the Hirzebruch-Jung continued fraction of $529/252$, where $(529,252)=1$ and the boundary of $P'$ is the lens space $L(529,252)$. Let us label the spheres from left to right as $u'_1, \cdots, u'_{12}$. 
\begin{figure}
\begin{center}
\scalebox{0.95}{\includegraphics{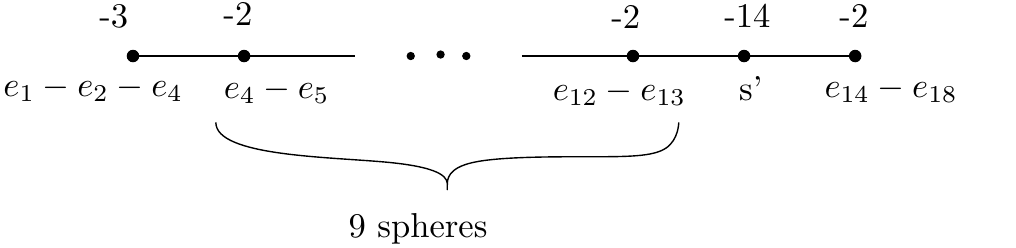}}\caption{Generalized rational blow down plumbing of length 12}
\label{6}
\end{center}
\end{figure} 
Note that this plumbing is symplectically embedded in $(\mathbb{F}_2 \# 12\CPb) \# 5\CPb = \CP \# 18\CPb$. Then we rationally blow down this plumbing in $\CP \# 18\CPb$, and call the resulting symplectic manifold $\mathcal Z$. Next, we will show that $\mathcal Z$ is an exotic copy of $\CP \# 6\CPb$. As above, we first show that $\mathcal Z$ is homeomorphic to $\CP \# 6\CPb$.

Let $\mathcal Z = (\CP \# 18\CPb - P') \cup B'$, where $B'$ is the rational homology ball whose boundary is the lens space which also bounds $P'$. We contract the generator of $\pi_1(\partial P')$ along the sphere $F-e-e_3 = e_2- e_3$. Note that $( e_2- e_3)$ intersects $P'$ but was not used in the rational blow down plumbing $P'$. On the other hand, we have the surjection $\pi_1(\partial B') \twoheadrightarrow \pi_1(B')$. Thus, $\mathcal Z$ is simply connected by Van Kampen's theorem. As above, we compute $e, \sigma$ of $\mathcal Z$, and by Freedman's classification theorem, we conclude that $\mathcal Z$ is homeomorphic to $\CP \# 6\CPb$

Next, we will show that $\mathcal Z$ is not diffeomorphic to $\CP \# 6\CPb$. Let 
\begin{equation}
w'= a'h-b'_1e_1-\cdots-b'_{18}e_{18}     
\end{equation}
be the cohomology class of a symplectic strucutre on $\CP \# 18 \CPb$ with $a'>b'_1>\cdots>b'_{18}>0$ and $a'>b'_1+\cdots+b'_{18}$. Let $K'$ be the canonical class of $\CP \# 18 \CPb$.  We have 
\begin{equation}
K'= -3h+e_1+\cdots+e_{18}. 
\end{equation}
Next, we compute
\begin{equation*}
\begin{aligned}
K' \cdot u'_1 =& (-3h+e_1+\cdots+e_{18}) \cdot (e_1-e_2-e_4) =1\\
K' \cdot u'_i =& 0,\;\;\; i = 2,\dots,10,12\\
K' \cdot u'_{11} =& (-3h+e_1+\cdots+e_{18}) \cdot\\
&(8h-4e_1-2e_2-\cdots-2e_{12}-e_{13}-2e_{14}-..-2e_{17}-e_{18})\\
&=12.
\end{aligned}
\end{equation*}
Let $\gamma'_1,\cdots,\gamma'_{12}$ be the basis of $H^2(P',\mathbb{Q})$ which is dual to $u'_1,\cdots,u'_{12}$. Using the adjunction formula,
\begin{eqnarray*}
K'|_{P'} &=& \sum_{i=1}^{12}(K' \cdot u'_i) \gamma'_i\\
&=& \gamma'_1+12 \gamma'_{11}.
\end{eqnarray*}
Then we calculate the restriction of the symplectic class $w'$ on $P'$
\begin{eqnarray*}
w'|_{P'} &=& \sum_{i=1}^{12}(w' \cdot u'_i) \gamma'_i\\
&=& (b'_1-b'_2-b'_4) \gamma'_1+ (b'_4-b'_5) \gamma'_2 + (b'_5-b'_6) \gamma'_3 + (b'_6-b'_7) \gamma'_4\\
&+& (b'_7-b'_8) \gamma'_5 + (b'_8-b'_9) \gamma'_6 + (b'_9-b'_{10}) \gamma'_7 + (b'_{10}-b'_{11}) \gamma'_8\\
&+&(b'_{11}-b'_{12}) \gamma'_9 + (b'_{12}-b'_{13}) \gamma'_{10}+ (b'_{14}-b'_{18}) \gamma'_{12}\\
&+&(8a'-4b'_1-2b'_2-\cdots-2b'_{12}-b'_{13}-2b'_{14}\cdots-2b'_{17}-b'_{18})\gamma'_{11}.
\end{eqnarray*}
Let $M'$ be the intersection matrix for the plumbing $P'$:
$$\quad
M'=
\setcounter{MaxMatrixCols}{12}
\begin{bmatrix} 
-3 & 1 & 0 &0 &0 &0 &0 &0 &0 &0 &0 & 0 \\
1    &-2 &1 & 0 &0 &0 &0 &0 &0 &0&0&0\\
0&1&-2 &1 & 0 &0 &0 &0 &0 &0&0&0\\
0&0&1&-2 &1 & 0 &0 &0 &0 &0&0&0\\
0&0&0&1&-2 &1 & 0 &0 &0 &0&0&0\\
0&0&0&0&1&-2 &1 & 0 &0 &0&0&0\\
0&0&0&0&0&1&-2 &1 & 0 &0&0&0\\
0&0&0&0&0&0&1&-2 &1 & 0 &0&0\\
0&0&0&0&0&0&0&1&-2 &1&0&0\\
0&0&0&0&0&0&0&0&1&-2&1&0\\
0&0&0&0&0&0&0&0&0&1&-14&1\\
 0&0&0&0&0&0&0&0&0&0&1&-2
\end{bmatrix}$$
To find $K'|_{P'} \cdot w'|_{P'}$, we need to read the first and the eleventh columns of $M'^{-1}$ which are $[-1/529 (252, 227, 202, 177, 152, 127, 102, 77, 52, 27,2,1)]$ and \\
$[-1/529 (2, 6, 10, 14, 18, 22, 26, 30, 34, 38, 42, 21)]$, respectively. Hence, we find 
\begin{eqnarray*}
K'|_{P'} \cdot w'|_{P'} &=& -1/529 [4048a' - 1748b'_1- 1288b'_2 - 989(b'_4+b'_5+\cdots+b'_{13})\\
&-&759(b'_{14}+b'_{18})-1012(b'_3+b'_{15}+b'_{16}+b'_{17})].
\end{eqnarray*}
We also have $K' \cdot w' = -3a' + b'_1 + \cdots + b'_{18}$, so we have
\begin{eqnarray*}
K'|_{\mathcal Z} \cdot w'|_{\mathcal Z} &=& K' \cdot w' - K'|_{P'} \cdot w'|_{P'}\\
&=& 1/529 [2461a' -1219b'_1 -759b'_2 -460(b'_4+\cdots+b'_{13})\\ 
&-&230(b'_{14}+b'_{18})-483(b'_3+b'_{15}+b'_{16}+b'_{17})]\\
&>&0.
\end{eqnarray*}
This shows that $\mathcal Z$ is not diffeomorphic to $\CP \# 6\CPb$, since $\CP \# 6\CPb$ does not admit a symplectic structure with $K' \cdot w' >0$ as explained above.

Furthermore, minimality of $\mathcal Z$ follows using the methods developed in \cite{{OzsSz}}. 
\end{proof}

\subsection{Fibration V - V*}\label{VV*}

Let us take the fibration of type (V - V*). That is to say we have $f: \CP \# 13\CPb \rightarrow \mathbb{CP}^1$ relatively minimal fibration of genus $2$, having two singular fibers of types V and V* (\cite{Gong}). The following lemma is similar to Lemma 3.2 in \cite{Kit1}.

\begin{lemma} For the relatively minimal, genus 2 fibration $f: \CP \# 13\CPb \rightarrow \mathbb{CP}^1$ with two singular fibers V and V*, there exists a birational morphism $\mu:  \CP \# 13\CPb  \rightarrow \mathbb{F}_3$ such that the images of the fiber V* and a generic fiber $F$ of $f$ under $\mu$ form the pencil as at the beginning of the Figure~\ref{five}. \end{lemma}

\begin{proof}
We will use Theorem 2.2 in \cite{Kit1} which stated as Theorem 2.9 above. Note that $\CP \# 13\CPb$ is a rational surface, $f$ is a relatively minimal fibration of genus $g=2$, and the Picard number 
\begin{equation*}
\rho(\CP \# 13\CPb) = 10-K_{\CP \# 13\CPb}^2 = b_2(\CP \# 13\CPb) = 14
\end{equation*}
(For more discussion on Picard numbers the reader may see e.g. \cite{SaSa}, proof of Theorem 2.8). Therefore, $\rho(\CP \# 13\CPb)$ equals to $4g+6$, since we have $g=2$. Thus, by Theorem~\ref{main1} above, there exists a birational morphism $\mu:  \CP \# 13\CPb  \rightarrow \mathbb{F}_3$. 

By the afore-mentioned theorem, part i), we have $(\mu_*F) = 2C_{\infty}= 2C_0+6F$ where $F$ is the generic fiber of the genus two fibration $f$. Let us denote $(\mu_*F) = 2C_{\infty}= 2C_0+6F$ by $B$.

By Theorem~\ref{main1} once again, $f$ has at least one $(-1)$ section and it intersects a component of multiplicity 1 (for the latter fact, see \cite{Kit1}, proof of Lemma 3.2). Fiber V* has only two components with multiplicity one, $a$ and $k$ as in Figure \ref{V*}. Thus the section $s$ intersects either the component $a$ or $k$ in V*. Without loss of generality, let $s$ intersect the component $a$ as in Figure \ref{V*}. Then from the proof of Lemma 3.2 in \cite{Kit1}, it easily follows that $\mu$ is the birational morphism contracting the section $s$, and components $a$, $b$, $c$, $d$, $e$, $f$, $g$, $h$, $i$, $j$, $k$ in turn (see Figure \ref{V*}). Note that $(\mu_*l)^2 =3$, and $(\mu_*l) = C_{\infty}$ is the $(+3)$ section of $\mathbb{F}_3$. Hence our claim follows.
\end{proof}
 
\begin{figure}
\begin{center}
\scalebox{0.80}{\includegraphics{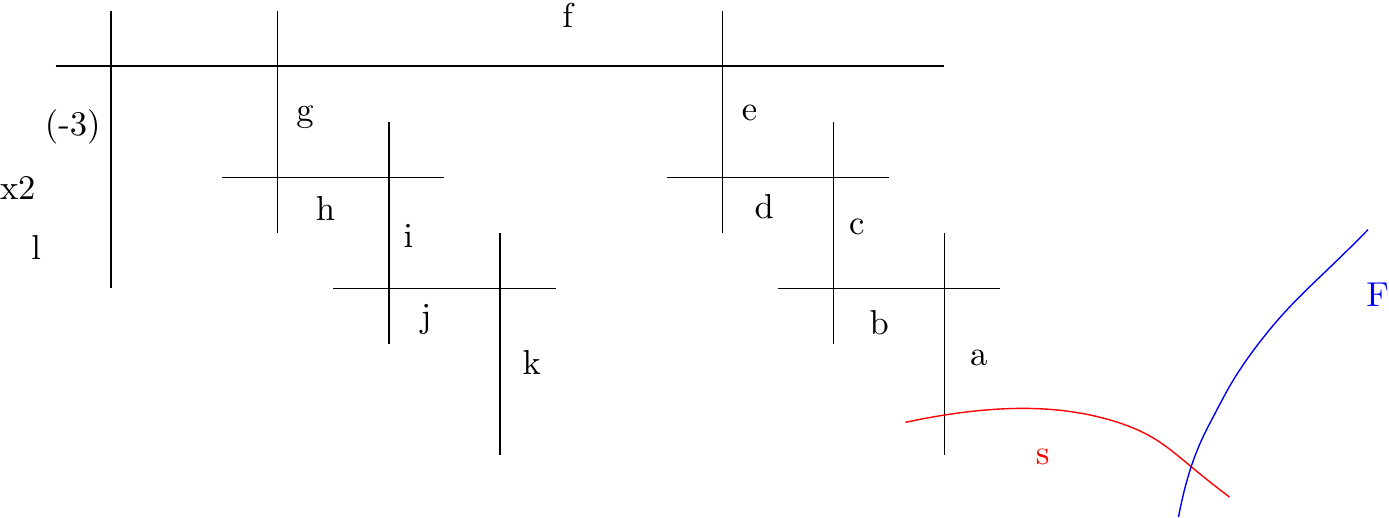}}\caption{V* and a generic fiber}
\label{V*}
\end{center}
\end{figure} 
 
Next, starting from this pencil we will do blow ups. By carefully computing the multiplicities at each step, we will find the homology classes of the components of fiber V*. We note that we will also obtain the multiplicities of each component of V* which match the ones on the list of \cite{NamU}.

\begin{theorem} 
There exists an irreducible symplectic 4-manifold homeomorphic but non-diffeomorphic to $\CP \# 7\CPb$, obtained from the total space of a genus two fibration with two singular fibers of types V and V*, using a combination of the $2$-spherical deformations, symplectic blowups, and the (generalized) rational blowdown surgery.
\end{theorem}

\begin{proof}
In this construction, we will work in $\mathbb{F}_3 \cong \mathbb{CP}^2 \# \CPb$. Let us take the $(+3)$ section $C_\infty$ with multiplicity two in $\mathbb{F}_3$, and denote it by $A$. The second component $B$ of our pencil is an irreducible algebraic curve in the linear system $|2C_\infty=2C_0+6F|$, where $C_0$ is the $(-3)$ section and $F$ is the fiber.


Notice that genus of $B=2C_\infty$ is 2 and $2C_\infty \cdot C_\infty =6$. From the previous lemma we have it's one point of multiplicity 6. Hence we represent $A$ and $B$ curves as at the beginning of Figure \ref{five} where $A$ is the black and $B$ is the blue curve. Since $A$ and $B$ represent the same class in homology, they define a Lefschetz pencil with a nonempty base locus. In Figure \ref{five} we denote it by the red point and its multiplicity by $\times 6$.
\begin{figure}
\scalebox{0.70}{\includegraphics{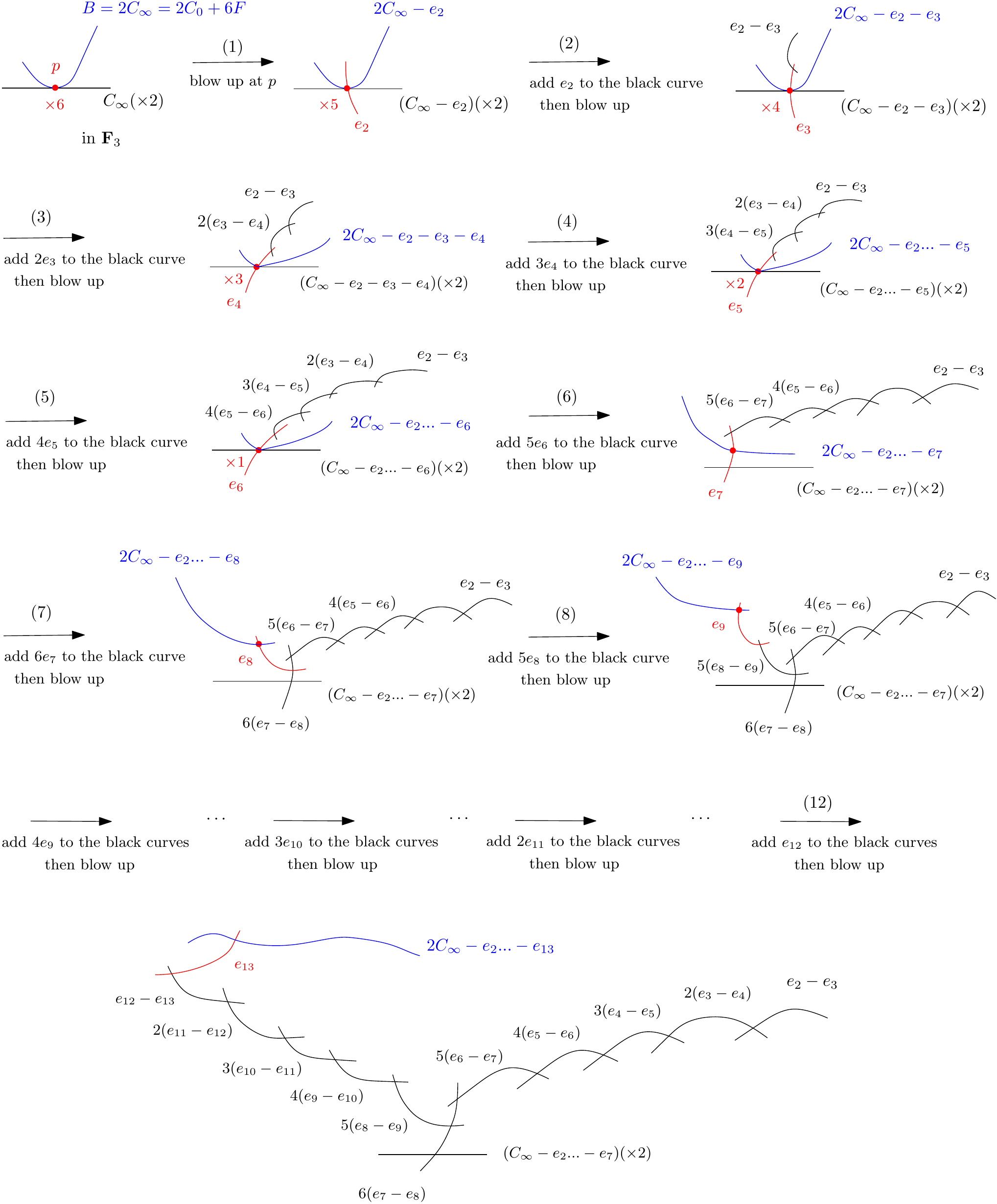}}\caption{Fibers V and V*}
\label{five}
\end{figure} 
We blow up that point and denote the exceptional divisor by $e_2$ in $\mathbb{F}_3 \# \CPb$. We proceed as in the previous constructions. Namely, after each blow up we reset the $A$ and $B$ curves so that we obtain a Lefschetz pencil. Then we blow up the new intersection point of $A$ and $B$. We pursue this process until $A$ and $B$ curves are separated and have equal homology classes. We showed each step in Figure \ref{five} with the homology classes and multiplicities of every irreducible component. At the end, we remark that the total homology class of the first curve (the black part) is $2C_\infty -e_2- \cdots -e_{13}$ which is the same as the homology class of the blue curve. Let us call this blue curve at the end $\tilde B$ for which we have $\tilde B= 2C_\infty -e_2- \cdots -e_{13}= 2(2h-e_1) - e_2 - \cdots -e_{13}=4h-2e_1-e_2- \cdots -e_{13}$, and so it is of self intersection 0. We note that the resulting black curve is the fiber V* in Namikawa-Ueno's list. Hence we attain a configuration as shown in the last step of Figure \ref{five} that is symplectically embedded in $\mathbb{F}_3 \# 12\CPb \cong \CP \# 13\CPb$. 

The fiber V* is the dual of the fiber V and the monodromy of the latter is $a_1a_2a_3a_4a_5$ in $\Gamma_2$ \cite{Ish2}. As we showed above we have:
\begin{equation}
(a_1a_2a_3a_4)a_5 = (a_4^{-1}a_1a_3a_4)(a_4^{-1}a_3^{-1}a_2a_4a_3a_4)a_5
\end{equation}  

so as before we make the same deformation as shown in Figure \ref{cusp}. The resulting two fibers, with $2$ nodes on each, are in $\CP \# 13\CPb$. As in the previous cases, we blow them up twice (cf. Figure \ref{cusp}) and resolve their proper transforms $(\tilde B-2e_{14}-2e_{15}), (\tilde B-2e_{16}-2e_{17})$ with the section $e_{13}$. Again we obtain the class
\begin{eqnarray*}
s&:=& (\tilde B-2e_{14}-2e_{15})+(\tilde B-2e_{16}-2e_{17})+e_{13}\\ 
&=& 8h - 4e_1-2e_2 -2e_3-\cdots-2e_{12}-e_{13}-2e_{14}-\cdots-2e_{17}, 
\end{eqnarray*}
of square $-13$. Note that we used the section $e_{13}$ in Figure \ref{five} to construct $s$. Hence we have that the fiber V* intersects the class $s$ once in $\CP \# 17\CPb$. This gives us the same plumbing $P$ of length ten as in Figure \ref{bdown}, symplectically embedded in $\CP \# 17\CPb$. We rationally blow it down. Moreover, the canonical class of $\mathbb{F}_3$ is $K_{\mathbb{F}_3}= -2C_{\infty}+F$ as we showed in the introduction. So we have 
\begin{equation}
K_{\mathbb{F}_3}= -2C_{\infty}+F = -2(2h-e_1)+(h-e_1)= -3h+e_1.
\end{equation}
Therefore, in $(\mathbb{F}_3 \# 12\CPb) \#4\CPb \cong (\CP \# 13\CPb)\#4\CPb  \cong  \CP \# 17\CPb$ we have
\begin{equation}
K= -3h+e_1+e_2+ \cdots +e_{17}
\end{equation}
and the fiber class is
\begin{equation}
F= 4h-2e_1-e_2- \cdots -e_{13}.
\end{equation} 
We show that the resulting manifold after the rational blow down is an exotic copy of $\CP \# 7\CPb$.  The proof will follow the same lines of computation as in the first case above. 
\end{proof}

\subsection{Small exotic $4$-manifolds from genus two fibrations of type 5(IX-2), and of type (IX-2)-2(IX-4)} 

In what follows, we will make use of the explicit constructions of genus two fibrations on $K3\#2\,\CPb$ given in subsection~\ref{pencilK3} to construct exotic copies of $3\,\CP\#k\,\CPb$ for $k = 16, 17, 18, 19$.  Our first construction will be applied to such a fibration with five singular members of type $IX-2$, and we will employ the operations of symplectic resolution and rational blowdowns along $-4$ spheres. On the other hand, our second construction will be applied to fibrations with one singular fiber of type $IX-2$ and two singular fibers of type $IX-4$, where we will use will use $2$-nodal spherical deformations, symplectic blowups, symplectic resolution, and the (generalized) rational blowdown surgery.

\begin{theorem} 
Let $M$ be one of the following 4-manifolds 
\begin{enumerate}
\item $3\CP \# 16\CPb$ 
\item $3\CP \# 17\CPb$ 
\item $3\CP \# 18\CPb$ 
\item $3\CP \# 19\CPb$ 
\end{enumerate}
Then there exists an irreducible symplectic 4-manifold homeomorphic but non-diffeomorphic to $M$, obtained from the total space of a genus two fibration with five singular fibers of type IX-2, using a combination of symplectic resolution and rational blowdown surgery along $-4$ spheres.
\end{theorem}
\begin{proof} We start with the genus two fibration structure on $K3\#2\,\CPb$ with five singular fibers of type IX-2 explained in subsection~\ref{pencilK3} (see Lemma~\ref{pencils} and it's proof). Recall that a singular fiber of type IX-2 is the union of three smooth rational curves $F_j$, $G_j$, $E_j$ (for $1 \leq j \leq 5$) passing through a single point, say $x_{j}$. Moreover, $F_j$ is tangent to $G_j$ at $x_j$, and self-intersections of these rational curves are given as follows: $F_j^{2} = G_j^{2} = -3$, and $E_j^{2} = -2$. For each of these five singular fibers, let us symplectically resolve the intersection point $x_j$ of $F_j$ and $G_j$ (where $1 \leq j \leq 5)$ to obtain a symplectic sphere $S_j = F_j + G_j$ of self-intersection $-4$ in $K3\#2\,\CPb$. Each of these $-4$ spheres $S_j$ has a dual sphere, they are all disjoint from each other. Let $M(i)$ denote the symplectic $4$-manifold gotten from $K3\#2\,\CPb$ by performing the rational blowdown surgery along disjoint $-4$ spheres $S_1$, $\cdots$, $S_i$,  where $1 \leq i  \leq 5$. 

Let us first verify that $M(i)$ is homeomorphic to $3\CP \#(21-i)\CPb$. This is a repated application of Lemma~\ref{rbd} and Freedman's classification theorem (Theorem (1.5) of \cite{F}). For the sake of completeness and clarity, let us spell out the details. Let $P_2$ be a tubular neighborhood of the sphere $S_1$ with self-intersection $-4$ in $K3\#2\,\CPb$. We have $M(1) = (K3 \# 2\CPb - P_{2}) \cup B_{2}$, where $B_{2}$ is a rational homology ball whose boundary is the lens space $L(4,1)$, which also bounds $P_{2}$. We can contract the generator of $\pi_1(\partial P_{2})$ using the sphere dual to $S_1$. Since we have the surjection $\pi_1(\partial B_{2}) \twoheadrightarrow \pi_1(B_{2})$, $M(1)$ is simply connected by Van Kampen's Theorem. By applying the formulas of Lemma~\ref{rbd}, we compute
\begin{eqnarray*}
e (M(1))&=& e(K3 \# 2\CPb) - e(P_{2}) + e(B_{2})\\
&=& 26-2+1\\
&=&25,\\
\sigma (M(1))&=& \sigma (K3 \# 2\CPb) - \sigma (P_{2})\\
&=& -18-(-1)\\
&=&-17.
\end{eqnarray*}
 A repeated application of Van Kampen's theorem shows that $M(i)$ is simply connected, and we have
\begin{eqnarray*}
e (M(i))&=& e(K3 \# 2\CPb) - ie(P_{2}) + ie(B_{2})\\
&=& 26-2i+i\\
&=&26-i,\\
\sigma (M(i))&=& \sigma (K3 \# 2\CPb) - i\sigma (P_{2})\\
&=& -18-(-i)\\
&=&-18+i.
\end{eqnarray*}
$M(i)$ contains curves with an odd self-intersection for any $1 \leq i \leq 5$, so it is a simply connected non-spin $4$-manifold. Thus, we can conclude by Freedman's classification theorem that $M(i)$ is homeomorphic to $3\CP \#(21-i)\CPb$. 

Next, we will verify that $M(i)$ is not diffeomorphic to $3\CP \#(21-i)\CPb$. By the blow up formula for the Seiberg-Witten function, we have $SW_{K3 \# 2\CPb} = SW_{K3} \displaystyle\prod_{k=1}^{2}(e^{e_k} +e^{−e_k}) = (e^{e_1} + e^{-e_1})(e^{e_2} + e^{-e_2})$, where $e_k$ is an exceptional divisor class resulting from the $k$-th blow up (of the base points of genus two pencil) in $K3 \# 2\CPb$. Consequently, the set of basic classes of $K3 \# 2\CPb$ is given by $\pm e_1 \pm e_2$, and the values of Seiberg-Witten invariants on $\pm e_1 \pm e_2$ are $\pm 1$. It is routine to prove that after performing a single rational blowdown operation in $K3 \# 2\CPb$, along the sphere $S_1$, the resulting symplectic $4$-manifold $M(1)$ is diffeomorphic to $K3 \# \CPb$. This manifold has a pair of basic classes $\pm K_{M(1)}$, which descend from the top classes $\pm(e_1 + e_2)$ of $K3 \# 2\CPb$. By applying theorems from section~\ref{rb}, we completely determine the Seiberg-Witten invariants of $M(i)$ using the basic classes and invariants of $K3 \# \CPb$: Up to the sign the symplectic $4$-manifold $M(i)$ has only one basic class which descends from the canonical class of $K3 \# \CPb$. By Taubes theorem \cite{T}, the value of the Seiberg-Witten function on these classes $\pm K_{M(i)}$ evaluates as $\pm 1$. By applying the connected sum theorem for the Seiberg-Witten invariant, Seiberg-Witten function is trivial for $3\CP \#(21-i)\CPb$. Thus, we have shown that $M(i)$ is not diffeomorphic to $3\CP \#(21-i)\CPb$. Using Seiberg-Witten basic classes of $M(i)$, it is easy to verify that $M(i)$ is a minimal symplectic $4$-manifold when $i \geq 2$. This follows from the the fact that for $i \geq 2$, $M(i)$ has no two basic classes $K$ and $K'$ such that $(K - K')^{2} = -4$. Since symplectic minimality implies irreducibility for simply-connected $4$-manifolds with $b_2^{+} > 1$, it follows that $M(i)$ is also smoothly irreducible when $i \geq 2$.
\end{proof}

The following theorem will use the $2$-nodal spherical deformation introduced in Section~\ref{Nodal}. 

\begin{theorem} 
There exists an irreducible symplectic 4-manifold homeomorphic but non-diffeomorphic to $3\CP \# 17\CPb$, obtained from the total space of a genus two fibration with one singular fiber of type IX-2 and two singular fibers of type IX-4, using a combination of the $2$-nodal spherical deformation, symplectic blowups, symplectic resolution, and the rational blowdown surgery along $P_{6}$.
\end{theorem} 

\begin{proof} We start with a genus two fibration structure on $K3\#2\,\CPb$ with one singular fiber of type IX-2 and two singular fibers of type IX-4 given in subsection~\ref{pencilK3} (see Lemma~\ref{pencils}). Recall that the singular fiber of type IX-4 consists of $9$ rational curves $E_j$, ($1 \leq j \leq 7$), $F$, $G$. The intersection matrix corresponding to the first $7$ smooth rational curves $E_j$ has type $D_7$. This means that nonzero entries of this matrix are given by $E_1 \cdot  E_2 = E_2\cdot E_3= E_3 \cdot E_4 = E_4 \cdot E_5= E_5 \cdot E_6 = E_5 \cdot E_7= 1$, $E_{j}^{2} = -2$. In this case, $F$ and $G$ are disjoint and each of them meets one componet of $D_7$. We will assume that $F$ meets $E_6$ and $G$ meets $E_7$. Moreover, it is easy to check that two sphere sections $e_1$ and $e_2$, resulting from two blow ups of the base points of the genus two pencil in $K3$ surface, hit only one of the components $F$ and $G$.  Let us assume that $e_1 \cdot F = e_2 \cdot G = 1$.

The monodromy relation corresponding to the above splitting of the singular fibers in $K3\#2\,\CPb$ is given by the following word  
\begin{eqnarray*}
1 = (a_1a_2a_3a_4a_5^2)^{5} &=& (a_1a_2a_3a_4a_5^2) (a_1a_2a_3a_4a_5^2)^{2}(a_1a_2a_3a_4a_5^2)^{2}, 
\end{eqnarray*}
where the monodromy $(a_1a_2a_3a_4a_5^2)$ corresponds to a type IX-2 singular fiber, and each $(a_1a_2a_3a_4a_5^2)^{2}$ corresponds to a type IX-4 singular fiber \cite{Ish, Ish2, NamU}. By Lemma~\ref{f1} the word $(a_1a_2a_3a_4a_5^2)$ can be deformed to contain two disjoint $2$-nodal spherical singular fibers and two nodal fibers in $K3\#2\,\CPb$. We will only use one of these $2$-nodal spherical singular fibers to build a negative-definite plumbing tree for our rational blowdown. Let us blow up two nodes of one of these $2$-nodal spherical singular fibers, say $B$, and symplectically resolve the intersection point of its proper transform $\tilde B-2e_{3}-2e_{4}$ with the sphere section $e_{1}$ and intersection point of the sphere section with $-3$ sphere $F_i$. The resulting symplectic sphere $S$ has self-intersection $-8$. Notice that $S$ together with the spheres $E_7$, $E_5$, $E_4$, $E_3$ forms a negative-definite plumbing tree $P_{6}$ symplectically embedded in $K3\#4\,\CPb$.

Let $M(3,17)$ denote the symplectic $4$-manifold we get from $K3\#4\,\CPb$ by performing the rational blowdown surgery along  $P_{6}$. Similarly as before, we show $M(3,17)$ is simply connected: we contract the generator of $\pi_1(\partial P_{6})$ using the sphere $E_6$ that was not used in the rational blow down plumbing $P_{6}$. Using the formulas, we compute $e, \sigma$ of $M(3,17)$, and by Freedman's classification theorem, we conclude that $M(3,17)$ is homeomorphic to $3\CP \# 17\CPb$. Also, we verified non-triviality of Seiberg-Witten invariants of $M(3,17)$.

\end{proof}

\begin{remark} Our above constructions of exotic $3\CP \# k\CPb$ from genus two fibrations are certainly not optimal. We can also construct two disjoint rational blowdown plumbings $P_{8}$ in $K3\#8\,\CPb$, which yields an exotic copy of $3\CP \# 13\CPb$. Such a construction can be achieved by using the second $-1$ sphere section $e_2$ and the second $2$-nodal spherical singular fiber. By blowing up intersection point of the first $2$-nodal spherical singular fiber with $e_2$, and the second $2$-nodal spherical singular fiber with $e_1$, we can form two disjoint configurations of $P_{8}$ in $K3\#8\,\CPb$. It is also not difficult to construct rational blowdown plumbings of shorter length from genus two singular fibers in Namikawa and Ueno's list and use them to obtain exotic copies of $3\CP \# k\CPb$ for $k = 14, 15, 16, 18, 19$. 
\end{remark}


\begin{thebibliography}{99}

\bibitem{A} A. Akhmedov, \textit{Construction of exotic smooth structures}, Topology Appl., \textbf{154} (2007), 1134-1140.

\bibitem{akhmedov} A. Akhmedov,
\textit{Small exotic\/ $4$-manifolds},
Algebr. Geom. Topol., \textbf{8} (2008), 1781-1794.

\bibitem{A1} A. Akhmedov,
\textit{Construction of symplectic cohomology $\mathbb{S}^{2}\times \mathbb{S}^{2}$},
G\"{o}kova Geometry and Topology Proceedings, {\bf 14} (2007), 36-48.

\bibitem{ABBKP} A. Akhmedov, S. Baldridge, R. \.{I}. Baykur, P. Kirk and B. D. Park,
\textit{Simply connected minimal symplectic\/ $4$-manifolds with
signature less than\/ $-1$},
J. Eur. Math. Soc., {\bf 1} (2010), 133-161.

\bibitem{ABP} A. Akhmedov, R. \.{I}. Baykur and B. D. Park,
\textit{Constructing infinitely many smooth structures on small\/ $4$-manifolds},
J. Topol., \textbf{1} (2008), 409-428.

\bibitem{AP1} A. Akhmedov and B. D. Park,
\textit{Exotic smooth structures on small\/ $4$-manifolds},
Invent. Math., \textbf{173} (2008), 209-223.

\bibitem{AP2} A. Akhmedov and B. D. Park,
\textit{Exotic smooth structures on small\/ $4$-manifolds with odd signatures},
Invent. Math., \textbf{181} (2010), 577-603.

\bibitem{AkP} A. Akhmedov, J.-Y. Park {\em Lantern substitution and new symplectic 4-manifolds with $b_{2}^{+} = 3$}, Math. Res. Lett., {\bf 21} (2014), no. 1, 1-17.

\bibitem{A2} A. Akhmedov, N. Monden, {\em Constructing Lefschetz fibrations via daisy substitutions}, Kyoto J.  Math., {\bf 56}, No. 3 (2016), 501-529.


\bibitem{AsKo} T. Ashikaga, K. Konno, {Global and Local Properties of Pencils of Algebraic Curves}, Adv. Stud. Pure Math. {\bf 36}, 2002 Algebraic Geometry 2000, Azumino, 1-49.


\bibitem{Be} A. Beauville, {\em Complex Algebraic Surfaces}, London Math. Society Student Texts {\bf 34}, 2nd edition (1996).

\bibitem{Bol} O. Bolza, {\em On binary sextics with linear transformations into themselves}, Amer. J. Math., {\bf 10} (1988), 47-70.

\bibitem{CH} A. Casson and J. Harer, {\em Some homology lens spaces which bound homology rational ball}, Pacific. J. Math. {\bf 96} (1981), 23-36.

\bibitem{F} M. H. Freedman, {\em The topology of four-dimensional manifolds}, Jour. Diff. Geom. {\bf 17} (1982), 357-453.

\bibitem{FS1} R. Fintushel and R. Stern, {\em Rational blowdowns of smooth $4$-manifolds}, Jour. Diff. Geom. {\bf 46} (1997), 181-235.


\bibitem{FS3} R. Fintushel and R. Stern, {\em Double node neighborhoods and families of simply connected 4-manifolds with $b^{+} = 1$}, J. Amer. Math. Soc. {\bf 19} (2006), 171-180.



\bibitem{F2} T. Fuller, {\em Generalized Nuclei of Complex Surfaces}, Pacific. J. Math., {\bf 187} No. 2, 1999.

\bibitem{Gong} C. Gong, J. Lu, S.-L. Tan, \textit{On Families of Complex Curves over $\mathbb{P}^1$ with Two
Singular Fibers}, Osaka J. Math, {\bf 53} (2016), 83-99.

\bibitem{Gong2} C. Gong, J. Lu, S.-L. Tan, \textit{On The Classification and Mordell-Weil Groups of Families of Curves with Two Singular Fibers}, preprint, \\
 \url{http://math.ecnu.edu.cn/~jlu/7(alt)classification.pdf}

\bibitem{GS} R. Gompf and A. Stipsicz, {\em $4$-manifolds and Kirby calculus}, Grad. Stud. in Math. AMS (1999).


\bibitem{Hart} R. Hartshorne, {\em Algebraic Geometry}, Grad. Texts in Math., Springer, {\bf 52}, 1977.

\bibitem{Harts} R. Hartshorne, {\em Curves with High Self-Intersection on Algebraic Surfaces}, Publications mathematiques de l'I.H.E.S., {\bf 36} (1969), 111-125.

\bibitem{H} F. Hirzebruch, {\em \"{U}ber eine Klasse von einfachzusammenhängenden komplexen Mannigfaltigkeiten}, Math. Ann., {\bf 124} (1951), 77-86.

\bibitem{HJ} F. Hirzebruch, K. Janich, {\em Involutions and Singularities}, Proc. of the Int. Colloq. on Algebraic Geometry (Bombay 1968), S. 219-240, Oxford University Press 1969.

\bibitem{Ish2} M. Ishizaka, {\em Classification of The Periodic Monodromies of Hyperelliptic Families}, Nagoya Math. J., {\bf 174} (2004), 187-199.

\bibitem{Ish} M. Ishizaka, {\em Presentation of hyperelliptic periodic monodromies and splitting families}, Rev. Mat. Complut., {\bf 20} (2007), 483-495.

\bibitem{Ita} S. litaka, {\em On the degenerates of a normally polarized abelian variety of dimension 2 and an algebraic curve of genus 2 (in Japanese)}, Master degree thesis, University of Tokyo (1967). 


\bibitem{KS} \c{C}. Karakurt, L. Starkston, {\em Surgery along star-shaped plumbings and exotic smooth structures on $4$-manifolds}, Algebr. Geom. Topol., {\bf 16} (2016), 1585-1635.

\bibitem{Kit} S. Kitagawa, {\em Pencils of Genus Two Curves on Rational Surfaces}, (2010), \\
 \url{https://arxiv.org/pdf/1006.4372.pdf}.

\bibitem{Kit1} S. Kitagawa, {\em Extremal Hyperelliptic Fibrations on Rational Surfaces}, Saitama Math. J,  {\bf 30} (2013), 1-14.

\bibitem{Kod} K. Kodaira: {\em On compact analytic surfaces II}, Ann. Math., {\bf 78} (1963), 563-626.

\bibitem{KSB} J. Kollar, N. I. Shepherd-Barron, {\em Threefolds and deformations of surface singularities}, Invent. Math., {\bf 91} (1988), no. 2, 299-338.

\bibitem{Kon} S. Kondo  {\em The Moduli Space of 5 Points on $P^1$ and $K3$ Surfaces}, Arithmetic and Geometry Around Hypergeometric Functions,  Progress in Mathematics book series (volume 260), 189-206.


\bibitem{LL}T. J. Li and A. Liu, {\em Symplectic Structure on Ruled Surfaces and A Generalized Adjunction Formula}, Math. Res. Lett., {\bf 2} (1995), 453-471.

\bibitem{TJ} T. J. Li, {\em Symplectic Parshin-Arakelov Inequality}, International Mathematics Research Notices, {\bf 18} (2000), 941-954.

\bibitem{LL2} T. J. Li, A. Liu, {\em Family Seiberg-Witten invariants and wall crossing formulas}, Communications in Analysis and Geometry, {\bf 9}, Number 4, 777-823, (2001).

\bibitem{Mich} M. Michalogiorgaki, {\em Rational blow-down along Wahl type plumbing trees of spheres}, Algebr.  Geom. Topol., {\bf 7} (2007), 1327-1343.


\bibitem{NamU2} Y. Namikawa, K. Ueno, {\em On Fibres in Families of Curves of Genus Two 1. Singular Fibres of Elliptic Type}, Number Theory, Algebraic Geometry and Commutative Algebra, in honor of Y. Akizuki, Kinokuniya, Tokyo, 1973, 297-371.

\bibitem{NamU} Y. Namikawa, K. Ueno, {\em The Complete Classification of Fibers in Pencils of Curves of Genus Two}, Manuscripta Math. {\bf 9} (1973), 143-186.

\bibitem{Nguen} K.-V. Nguen, {\em On Certain Mordell-Weil Lattices of Hyperelliptic Type on Rational Surfaces}, Journal of Mathematical Sciences, {\bf 102}, No. 2, (2000).

\bibitem{Ogg} A. P. Ogg, {\em On Pencils of Curves of Genus Two}, Topol., {\bf 5} (1966), 355-362.

\bibitem{OzsSz} P. Ozsvath, Z. Szabo, {\em On Park's Exotic Smooth Four-Manifolds}, “Geometry and topology of manifolds”, Fields Inst. Commun. 47, Amer. Math. Soc., (2005) 253-260.

\bibitem{P1} J. Park, {\em Seiberg-Witten invariants of generalized rational blow-downs}, Bull. Austral. Math. Soc., {\bf 56} (1997), 363-384.

\bibitem{P2} J. Park, {\em Simply connected symplectic $4$-manifolds with $b_2^+=1$ and $c_1^2=2$}, Invent. Math., {\bf 159} (2005), 657-667.

\bibitem{P3} J. Park, {\em Exotic smooth structures on $3\CP\#8\,\CPb$}, Bull. London Math. Soc. {\bf 39} (2007), 95-102.

\bibitem{PSS} J. Park, A. Stipsicz and Z. Szabo, {\em Exotic smooth structures on $\CP\#5\,\CPb$}, Math. Res. Lett., {\bf 12} (2005), 701-712.


\bibitem{SaSa} M. Saito, K. Sakakibara, {\em On Mordell-Weil lattices of higher genus fibrations on rational surfaces}, J. Math. Kyoto Univ. (JMKYAZ), {\bf 34-4} (1994), 859-871.

\bibitem{SS} S. Sakall{\i}, {\em Ph.D. Thesis}, University of Minnesota.

\bibitem{SS1} A. Stipsicz and Z. Szabo, {\em An exotic smooth structure on $\CP\# 6\,\CPb$ }, Geom. Topol., {\bf 9} (2005), 813-832.

\bibitem{SS2} A. Stipsicz and Z. Szabo, {\em Small exotic 4 manifolds with ${b_{2}}^{+} =3$}, Bull. London Math. Soc., {\bf 38} (2006), 501-506.

\bibitem{Sym} M. Symington, {\em Symplectic rational blowdowns}, J. Differential Geom. 50 (1998), no. 3, 505-518.

\bibitem{Tan} S.-T. Tan, {\em On the invariants of base changes of pencils of curves, II}, Math. Z., {\bf 222} (1996), 655-676.

\bibitem{T} C. Taubes, {\em The Seiberg-Witten invariants and symplectic forms}, Math. Res. Lett., {\bf 1} (1994), 809-822.

\bibitem{Ul} A. M. Uludag, {\em More Zariski pairs and finite fundamental group of curve complements},  Manuscr. Math. {\bf 106(3)} (2001), 271-277.

\bibitem{Wi} G. Winters, {\em On the Existence of Certain Families of Curves}, Amer. J. Math., {\bf 96(2)} (1974), 215-228.

\bibitem{Yeung} S-K. Yeung, \textit{Exotic Structures arising from Fake Projective Planes}, Sci. China. Math. \textbf{56(1)} (2013), 43-54.



\end{thebibliography}
\end{document}